\documentclass[11pt,letterpaper]{amsart}
\usepackage{amssymb,amsfonts, amsmath, color,graphicx, comment, mathtools, amsthm}
\usepackage{fontawesome}
\usepackage{tikz, tikz-cd, qtree} 

\usepackage{enumerate}

\usepackage[letterpaper,margin=1in]{geometry}

\usepackage{amsthm}
\usepackage[hidelinks]{hyperref}

\newtheorem{main}{Theorem}

\newtheorem{thm}{Theorem}[section]
\newtheorem{cor}[thm]{Corollary}
\newtheorem{lem}[thm]{Lemma}
\newtheorem{prop}[thm]{Proposition}
\newtheorem{thm*}{Theorem}

\theoremstyle{definition}

\newtheorem{fact}[thm]{Fact}

\theoremstyle{remark}
\newtheorem{remark}[thm]{Remark}

\numberwithin{equation}{section}

\newcommand{\Z}{\mathbb{Z}}
\newcommand{\R}{\mathbb{R}}
\newcommand{\Q}{\mathbb{Q}}
\newcommand{\C}{\mathbb{C}}
\newcommand{\N}{\mathbb{N}}

\newcommand{\Lang}{\mathcal{L}}

\newcommand{\cA
}{\mathcal{A}}
\newcommand{\cB}{\mathcal{B}}
\newcommand{\cM}{\mathcal{M}}
\newcommand{\cN}{\mathcal{N}}
\newcommand{\cR}{\mathcal{R}}
\newcommand{\cU}{\mathcal{U}}
\newcommand{\cV}{\mathcal{V}}

\renewcommand{\phi}{\varphi}

\newcommand{\scal}[1]{\left\langle #1 \right\rangle}
\newcommand{\norm}[1]{\left\lVert #1 \right\rVert}

\DeclareMathOperator{\qftp}{qftp}
\DeclareMathOperator{\tp}{tp}
\DeclareMathOperator{\Th}{Th}
\DeclareMathOperator{\re}{Re}

\DeclareMathOperator{\tr}{tr}
\DeclareMathOperator{\id}{id}

\DeclareMathOperator{\rank}{rank}
\DeclareMathOperator{\diag}{diag}
\newcommand{\cstar}{$\mathrm{C}^*$}

\begin{document}
	
	\title[Quantum expanders and  quantifier reduction]{Quantum expanders and quantifier reduction for tracial von Neumann algebras}
	\author[Farah, Jekel, and Pi]{Ilijas Farah, David Jekel, Jennifer Pi}

	\begin{abstract}
		We provide a complete characterization of theories of tracial von Neumann algebras that admit quantifier elimination. 
		We also show that the theory of a separable tracial von Neumann algebra $\cM$ is never model complete if its direct integral decomposition contains $\mathrm{II}_1$ factors $\mathcal{N}$ such that $M_2(\mathcal{N})$ embeds into an ultrapower of $\mathcal{N}$. The proof in the case of $\mathrm{II}_1$ factors uses an explicit construction based on random matrices and quantum expanders.
	\end{abstract}

	\address[Farah]{Department of Mathematics and Statistics,
		York University, \newline
		4700 Keele Street,
		Toronto, Ontario, Canada, M3J 1P3 \newline
		and Matemati\v cki Institut SANU \newline
		Kneza Mihaila 36, 11\,000 Beograd, p.p. 367 Serbia}
	\email{ifarah@yorku.ca}
	\urladdr{https://ifarah.mathstats.yorku.ca}
	\urladdr{https://orcid.org/0000-0001-7703-6931}
	
	\address[Jekel]{Fields Institute for Research in Mathematical Sciences, \newline
		222 College St, Toronto, Ontario, Canada, M5T 3J1}
	\email{djekel@fields.utoronto.ca}
	\address[Current address]{Department of Mathematical Sciences, University of Copenhagen \newline Universitetsparken 5, 2100 Copenhagen {\O}, Denmark}
	\email{daj@math.ku.dk}
	\urladdr{https://davidjekel.com}
	\urladdr{https://orcid.org/0000-0002-8580-5064}
	
	\address[Pi]{Department of Mathematics,
		University of California, Irvine, \newline
		410 Rowland Hall (Bldg.\# 400),
		Irvine, CA 92697-3875}
	\email{jspi@uci.edu}
	\address[Current address]{Mathematical Institute, University of Oxford, \newline 
		Radcliffe Observatory, Andrew Wiles Building, Woodstock Rd, Oxford OX2 6GG}
	\email{jennifer.pi@maths.ox.ac.uk}
	\urladdr{\href{https://sites.google.com/view/jpi314/home}{https://sites.google.com/view/jpi314/home}}
	\urladdr{https://orcid.org/0000-0003-1256-1086}
	
	
	\maketitle
	
	\section{Introduction}
	
	\subsection{On quantifier elimination}
	
	A common objection to the model-theoretic study of operator algebras \cite{GHS2013,FHS2013,FHS2014,FHS2014b} is that one needs to consider formulas with an arbitrarily large number of alternations of quantifiers. Since a typical human mind has difficulty parsing formulas such as $(\forall x_1)(\exists x_2)(\forall x_3)(\exists x_4)(\forall x_5)\psi(x_1,x_2,x_3,x_4 ,x_5)$ for a nontrivial relation $\psi$, it is natural to ask whether, for some theories, a given formula is equivalent to something simpler.  In particular, a theory $\mathrm{T}$ admits \emph{elimination of quantifiers} if every formula is equivalent modulo $\mathrm{T}$ to a quantifier-free formula (or in the metric setting, if every formula can be approximated by quantifier-free formulas).  
	
	Quantifier elimination has been isolated as a desirable property of theories from the very beginnings of model theory. Chang and Keisler \cite[\S 5.1]{chang1990model} wrote, ``Each time the method is applied to a new theory we must start from scratch in the proofs, because there are few opportunities to use general theorems about models. On the other hand, the method is extremely valuable when we want to beat a particular theory into the ground.’’ 
	Unfortunately--or fortunately, depending on one’s disposition--the only tracial von Neumann algebras whose theories admit quantifier elimination are of type I (i.e.\ a direct integral of matrix algebras), as the first author showed in \cite{Farah2023} (special cases were noted earlier in \cite{GHS2013}). Experts in operator algebras should not find it surprising that no II$_1$ factor has a theory that can be ``beaten into the ground''!
	
	Our first main result concerns which type I algebras admit quantifier elimination and confirms the conjecture from \cite{Farah2023}.
	
	\begin{main} \label{thm: main QE}
		Let $\cM = (M,\tau)$ be a WOT-separable tracial von Neumann algebra.  Then the following are equivalent.
		\begin{enumerate}[(1)]
			\item $\Th(\cM)$ admits quantifier elimination.
			\item $\cM$ is type I and any two projections $p$ and $q$ in $\cM$ with $\tau(p) = \tau(q)$ are conjugate by an automorphism of $\cM$.
		\end{enumerate}
	\end{main}
	
	Since the quantifier-free type of a projection is determined by its trace, condition (2) asserts that projections with the same quantifier-free type are conjugate by an automorphism.  We also give a more explicit description of when $\cM$ admits quantifier elimination in \S \ref{sec: QE tests}.  
	
	Special cases of tracial von Neumann algebras that admit quantifier elimination have been known for some time.  For instance, a diffuse commutative tracial von Neumann algebra corresponds to an atomless probability space, which Ben Ya'acov and Usvyatsov showed admit quantifier elimination in \cite[Example 4.3]{BYU2010} and \cite[Fact 2.10]{BY2012}. For further discussion, see \cite{BH2023} and \cite[\S 2.3]{JekelModelEntropy}.  The matrix algebras $M_n(\mathbb{C})$ also admit quantifier elimination thanks to the multivariate Specht's theorem \cite{Jing2015}.  Indeed, this result shows that two matrix tuples are unitarily conjugate if and only if they have the same $*$-moments under the trace, or equivalently, the same quantifier-free type.\footnote{However, $M_n(\mathbb{C})$ does not admit quantifier elimination as a $\mathrm{C}^*$-algebra (i.e.\ without the trace) since two nontrivial projections always have the same quantifier-free type but may not have the same type.}
	
	The question of quantifier elimination for $\mathrm{II}_1$ factors was studied in \cite[\S 2]{GHS2013}, which showed that the hyperfinite factor $\cR$ does not admit quantifier elimination, and this argument was observed in \cite{GH2023} to generalize to McDuff factors.  Furthermore, the results of \cite[\S 3]{GHS2013} imply that Connes-embeddable factors not elementarily equivalent to $\cR$ are not model complete, hence also do not admit quantifier elimination.  The first author \cite{Farah2023} extended this argument to refute quantifier elimination for $\mathrm{II}_1$ factors in general, and showed that tracial von Neumann algebras with a type $\mathrm{II}_1$ summand never admit quantifier elimination.  We also give another argument for this fact in Remark \ref{rem: alternate QE proof}.
	
	\subsection{On model completeness}
	
	Model completeness, introduced by Abraham Robinson, can be viewed as a poor person's version of quantifier elimination. A theory is \emph{model complete} if every embedding between its models is elementary.  While quantifier elimination means that every formula can be approximated by quantifier-free formulas, model completeness is equivalent to every formula being approximable by existential formulas (see \S \ref{subsec: QE and MC background}).  Thus, both quantifier elimination and model completeness are forms of quantifier reduction.
	
	Another characterization of model completeness for $\Th(\cM)$, under the assumption of the Continuum Hypothesis, is that for every separable $\cA$ and $\cB$ elementarily equivalent to $\cM$, every embedding $\cA \to \cB$ extends to an isomorphism $\cA^{\cU} \to \cB^{\cU}$ for some ultrafilter $\cU$ \cite[Corollary~16.6.5]{Fa:STCstar}.\footnote{The use of Continuum Hypothesis is, while necessary for this formulation, innocuous and removable at the expense of having a more complicated (but equally useful) formulation in terms of a back-and-forth system of partial isomorphisms between separable subalgebras of $\mathcal{A}$ and $\mathcal{B}$ that is $\sigma$-complete (see \cite[Theorem~16.6.4]{Fa:STCstar}).}  Operator algebraists will recognize this property as a generalization of the property of the hyperfinite II$_1$ factor $\cR$, that every embedding of $\cR$ into its ultrapower is unitarily equivalent to the diagonal embedding (the latter is elementary by \L o\' s's Theorem). By a standard ultrapower argument, this implies that every embedding of $\cR$ into a model of its theory, $\Th(\cR)$,  is elementary; this property was studied in \cite{AGKE2022} under the name of ``generalized Jung property.''  Note, however, that every embedding of $\cR$ into its ultrapower being elementary does not mean that $\mathcal{R}$ is model complete, since model completeness would require that every $\cM$ elementarily equivalent to $\cR$ also has the same property.
	
	Among tracial von Neumann algebras, type I algebras are model complete \cite{Farah2023} and algebras with a type II$_1$ summand are generally not model complete.  Indeed, the only possible model complete theory for Connes-embeddable $\mathrm{II}_1$ factors is $\Th(\cR)$ \cite[Proposition 3.2]{GHS2013}.  Moreover, \cite[Corollary 3.4]{GHS2013} showed that if the Connes embedding problem has a positive solution, then there is no model complete theory of a $\mathrm{II}_1$ factor; however, a negative solution of the Connes embedding problem was announced in \cite{JNVWY2020}, so the question of characterizing model complete theories of II$_1$ factors was still open.  It was conjectured in \cite{Farah2023} that tracial von Neumann algebras with a nontrivial type $\mathrm{II}_1$ summand are never model complete, and our second main theorem establishes this conjecture under a mild additional hypothesis that the $\mathrm{II}_1$ factors in the decomposition satisfy that $M_2(\mathcal{M})$ approximately embeds into $\mathcal{M}$.
	
	\begin{main} \label{thm: main MC}
		If $\cM$ is a $\mathrm{II}_1$ factor such that $M_2(\cM)$ embeds into $\cM^{\cU}$ for some ultrafilter $\cU$, then $\Th(\cM)$ is not model complete.
		
		More generally, let $\cM$ be a separable tracial von Neumann algebra with direct integral decomposition $\int_\Omega^{\oplus} (\cM_\omega,\tau_\omega)\,d\omega$.  Suppose that on a positive measure set, $\cM_\omega$ is a $\mathrm{II}_1$ factor such that $M_2(\cM_\omega)$ embeds into $\cM_\omega^{\cU}$ for some ultrafilter $\cU$.  Then $\Th(\cM)$ is not model complete.
	\end{main}
	
	The assumption that $M_2(\cM)$ embeds into an ultrapower of $\cM$ is closely related to \cite[Proposition 4.17]{GH2017}, and is immediate in several cases of interest.  For instance if $\cM$ is Connes embeddable this holds because $M_2(\cM)$ embeds into $\cR^{\cU}$ and hence into $M^{\cU}$ (of course, Theorem \ref{thm: main MC} in the Connes embeddable case also follows from \cite{GHS2013}). Another case where this condition is automatic is if $\cM$ is existentially closed in the class of $\mathrm{II}_1$ factors, since by definition there is an embedding of $M_2(\cM)$ into $\cM^{\cU}$ extending the diagonal embedding.  The condition also holds automatically if $\cM$ is McDuff, and more generally if its fundamental group is nontrivial; see \S \ref{sec: amplification}. 
	Although there are $\mathrm{II}_1$ factors such that $M_2(\cM)$ does not embed into $\cM$ \cite[Theorem C]{PopaVaes2022}, it is unknown at this point whether there exists any $\mathrm{II}_1$ factor such that $M_2(\cM)$ does not embed into $\cM^{\cU}$.  Since such an object would not be Connes-embeddable, it would no doubt be difficult to construct.  In \S \ref{sec: amplification}, we will discuss several conditions equivalent to $M_2(\cM)$ embedding into $\cM^{\cU}$.
	
	The proof of Theorem \ref{thm: main MC} is divided into two parts.  In the case of a $\mathrm{II}_1$ factor, we use a random matrix construction to create two tuples with similar behavior for their one-quantifier types, while their full types are distinguished by one having factorial commutant when the other does not.  In fact, this approach gives explicit sentences distinguishing their types (see \S \ref{sec: MC factor proof conclusion}).  The matrix construction shares some common ideas with \cite{Farah2023}, but also uses more substantial random matrix results such as Hastings's quantum expander theorem \cite{Hastings2007} and concentration of measure for random unitaries.  Thus, this is a first application of the combination of model theory and random matrix theory envisaged in \cite[\S 6]{JekelModelEntropy}. Already in \cite[\S 5]{Fa:Logic} it was predicted that deeper analysis of model theory of II$_1$ factors will necessarily involve free probability. 
	
	The extension to general tracial von Neumann algebras then requires two cases.  If the von Neumann algebra is a direct integral over a diffuse space, with fibers $M_\omega$, there is a direct argument to show the failure of model completeness when $M_2(\cM_\omega)$ embeds into $\cM_\omega^{\cU}$ (Lemma \ref{lem: diffuse MC}).  The remaining piece is the observation that if $\cM_1 \oplus \cM_2$ is model complete, then both $\cM_1$ and $\cM_2$ are model complete (Lemma \ref{lem: direct sum}).
	
	\subsection{Organization of this paper}
	
	In \S \ref{sec: preliminaries}, we recall background on tracial von Neumann algebras and continuous model theory, including specific tests for quantifier elimination and model completeness.
	In \S \ref{sec: QE proof}, we prove Theorem \ref{thm: main QE}, and in \S \ref{sec: QE tests}, we give several more explicit tests for quantifier elimination. In \S \ref{sec: MC factor proof}, we prove Theorem \ref{thm: main MC} in the case of $\mathrm{II}_1$ factors.  Then in \S \ref{sec: MC general case}, we prove the general case, relying on the fact that model completeness passes to direct summands (\S \ref{sec: MC direct sums}).  In the final section we give closing remarks:  in \S \ref{sec: topological} we discuss topological properties of theories of von Neumann algebras that have quantifier elimination or model completeness,  \S \ref{sec: amplification} is about the condition of $M_2(\cM)$ embedding to $\cM^{\cU}$, and \S \ref{sec: nontracial} is about quantifier elimination and model completeness in the non-tracial setting.
	
	\subsection*{Acknowledgements}
	
	We are grateful to the Fields Institute for hosting all three authors during the Thematic Program on Operator Algebras in Fall 2023 (IF as an organizer, DJ as a postdoc, and JP as a visitor).  We are grateful to Adrian Ioana for suggesting an argument that simplified the proof of Lemma~\ref{lem: second distance estimate}, and Ben Hayes for discussion of alternative proofs.  We thank Brent Nelson, Narutaka Ozawa, Isaac Goldbring, and Hiroshi Ando for discussions about type $\mathrm{III}$ factors.
	
	\section{Preliminaries} \label{sec: preliminaries}
	
	\subsection{Tracial von Neumann algebras}
	
	We assume familiarity with tracial von Neumann algebras, and recommend \cite{Ioana2023} for an introduction to the topic, as well as the standard reference books \cite{Blackadar2006,Dixmier1969,KadisonRingroseI,Sakai1971,TakesakiI,Zhu1993}.  In particular, we use the following notions and conventions:
	\begin{itemize}
		\item A \emph{tracial von Neumann algebra} is a finite von Neumann algebra with a specified tracial state. 
		\item The tracial state on $\cM$ will usually be denoted by~$\tau$ or $\tau_{\cM}$.
		\item The normalized trace on $M_n(\C)$ will be denoted by $\tr_n$.
		\item We also write $\norm{x}_2 = \tau(x^*x)^{1/2}$ when $x$ is an element of a tracial von Neumann algebra, and in particular when $x$ is a matrix, $\norm{x}_2 = \tr_n(x^*x)^{1/2}$ is the normalized Hilbert-Schmidt norm.
		\item The completion of $\cM$ with respect to $2$-norm is denoted $L^2(\cM)$.
		\item Inclusions and embeddings of tracial von Neumann algebras $\cN \subseteq \cM$ are assumed to be trace-preserving $*$-homomorphisms.
		\item If $\cN \subseteq \cM$, we denote by $E_{\cN}: \cM \to \cN$ the canonical conditional expectation; there is a unique conditional expectation that preserves the trace, and it is the restriction of the orthogonal projection $L^2(\cM) \to L^2(\cN)$.
	\end{itemize}
	
	\subsection{Continuous model theory} \label{sec: model theory prelims}
	
	We also assume some familiarity with continuous model theory, specifically model theory for metric structures; see e.g.\ \cite{BYBHU2008,Hart2023}.
	In particular:
	\begin{itemize}
		\item The structures under consideration are metric spaces, and the metric $d$ is one of the symbols in the language.  The structure can have multiple sorts; for instance, for a von Neumann algebra, there is one sort for each operator norm ball.
		
		\item Relation symbols are $\R$-valued, so in particular formulas will take values in $\R$ rather than evaluating to true/false.  The relation symbols and function symbols are required to be uniformly continuous across all models.
		
		\item Formulas are created in the usual recursive fashion with connectives from classical model theory replaced by continuous functions on $\R$, and the quantifiers $\forall$ and $\exists$ replaced with $\sup$ and $\inf$ (over appropriate bounded subsets of the von Neumann algebra). 
		
		\item For a language $\Lang$, and an $\Lang$-structure $\cM$, by the \textit{theory} of $\cM$ (denoted $\mathrm{Th}(\cM)$) we mean the set of all $\Lang$-sentences $\phi$ such that $\phi^\cM = 0$, except in \S \ref{sec: topological}, where it is more convenient to consider the theory as a bounded functional on the algebra of all formulas into $\mathbb R$. 
		
		\item For an $n$-tuple $\mathbf{a}$ coming from a structure $\cM$, the \textit{type} of $\mathbf{a}$ is the map $\tp^{\cM}(\mathbf{a}): \phi \mapsto \phi^\cM(\mathbf{a})$ which assigns to each $\Lang$-formula $\phi(x_1, \ldots, x_n)$ the value of $\phi^\cM(\mathbf{a})$. More generally, we say that any map $\mu$ which assigns a value $\phi(\mu) \in \R$ to each $\Lang$-sentence $\phi$ in $n$-variables is an $n$-type. For any fixed $n$, the space of all $n$-types is denoted $\mathbb{S}_n$.  Moreover, for a theory $\mathrm{T}$,  by $\mathbb{S}_n(\mathrm{T})$ we denote  the space of $n$-types that arise in models of $\mathrm{T}$.
		
		\item Quantifier-free formulas are those constructed recursively using connectives but no quantifiers.  The quantifier-free type $\qftp^{\cM}(\mathbf{a})$ is the restriction of $\tp^{\cM}(\mathbf{a})$ to quantifier-free formulas.
		
		\item The set $\mathbb{S}_n(\mathrm{T})$ is equipped with the \emph{logic topology}, which is the topology of pointwise convergence on $\Lang$-formulas, i.e. the weak$^*$-topology.  This makes $\mathbb{S}_n(\mathrm{T})$ into a compact Hausdorff space. Dually, each formula $\phi$ defines a continuous function on $\mathbb{S}_n(\mathrm{T})$.
		
		\item For any cardinal $\kappa$, we recall that a structure $\cM$ is \emph{$\kappa$-saturated} if every consistent type with parameters from a set $A \subseteq M$ with $|A| \leq \kappa$ is realized by some tuple $\mathbf{a}$ from $\cM$. (For operator algebraists, we note that a type is consistent with the theory of $\cM$ if it is in the weak$^*$-closure of the maps $\tp^{\cM}(\mathbf{a})$ for tuples $\mathbf{a} \in \cM$. Thus, countable ultraproducts of structures are countably saturated).
	\end{itemize}
	
	The language for tracial von Neumann algebras as metric structures was developed in \cite{FHS2014}, and other useful references include \cite[\S 2]{JekelCoveringEntropy} and \cite{GH2023}.  The sorts in this language are operator-norm balls, the functions are addition, multiplication, scalar multiplication, and adjoint, and the relation symbols are $\re \tr$ and the distance $d(x,y) = \norm{x - y}_2$.  All ultraproducts considered in this work are tracial; see \cite[\S 2.2]{FHS2014b} for a formal construction of tracial ultraproducts, and \cite[\S 16]{Fa:STCstar} or \cite[\S 2, \S 6]{Hart2023}  for more background on ultrafilters and ultraproducts in continuous model theory.
	
	\subsection{Definable Sets} \label{sec: definable}
	
	Lastly, in many arguments below we will need the notion of a definable set. These are sets that we are able to quantify over, without formally being a part of our language; see for instance \cite[Theorem 9.17]{BYBHU2008} and \cite[Definition~3.2.3 and Lemma~3.2.5]{FHLRTVW2021}. In particular, when $a$ is a definable element in some structure, then we can refer to it as if it were an interpretation of a  constant symbol in our language. We will use the following characterization of definable sets over a subset $A$ relative to a structure $\cM$, and refer the reader to \cite[\S 9]{BYBHU2008}, \cite[\S 2]{Goldbring2023spectralgap}, and \cite[\S 3]{FHLRTVW2021} for more information on definability.
	
	\begin{fact}
		Fix a structure $\cM$ and some subset $A \subseteq M$. Suppose $Z \subseteq M^n$ is a closed subset. Then $Z$ is a definable set in $\cM$ over $A$ if and only if for every $\epsilon > 0$, there exist $\delta > 0$ and a formula $\phi(x_1, \ldots, x_n)$, possibly using parameters from $A$, such that for any $\mathbf{x} \in M^n$,
		\[
		\phi^\cM(\mathbf{x}) < \delta \implies d(\mathbf{x}, Z) \leq \epsilon.
		\]
	\end{fact}
	If we say a set is definable in $\cM$, then we mean it is definable in $\cM$ over the empty set.
	
	\subsection{Quantifier elimination and model completeness} \label{subsec: QE and MC background}
	
	Recall that a theory $\mathrm{T}$ is said to admit \emph{quantifier elimination} if every $\Lang$-formula $\phi$ can be approximated uniformly across all models of $\mathrm{T}$ by quantifier-free $\Lang$-formulas.
	We will use the following characterization of quantifier elimination in terms of types.  A closely related statement for positive bounded logic is given in \cite[Proposition 14.21]{HIKO2003}.  The statement given here follows for instance from the proof of \cite[Lemma 2.14]{JekelModelEntropy}.
	
	\begin{lem} \label{lem: type QE}
		Let $\mathrm{T}$ be an $\Lang$-theory.  Then the following are equivalent:
		\begin{enumerate}
			\item $\mathrm{T}$ admits quantifier elimination.
			\item For every $n$ and every $\mu, \nu \in \mathbb{S}_n(\mathrm{T})$, if $\mu$ and $\nu$ agree on quantifier-free formulas, then $\mu = \nu$.
		\end{enumerate}
	\end{lem}
	
	There is an analogous characterization for model completeness, which can be regarded as a folklore result since it closely parallels what happens in discrete model theory (see e.g. \cite[Theorem 2.2]{Hirschfeld}).  Recall that an \emph{inf formula}, or \emph{existential formula}, is a formula obtained by preceding a quantifier-free formula with one or more $\inf$-quantifiers.

	\begin{lem} \label{lem: type MC}
		Let $\mathrm{T}$ be an $\Lang$-theory.  Then the following are equivalent:
		\begin{enumerate}
			\item $\mathrm{T}$ is model complete, i.e.\ if $\cM$ and $\cN$ are models of $\mathrm{T}$, then every embedding $\cM \to \cN$ is an elementary embedding.
			\item For every $n$ and every pair $\mu, \nu \in \mathbb{S}_n(\mathrm{T})$, if $\psi(\mu) \leq \psi(\nu)$ for every $\inf$-formula $\psi$, then $\mu = \nu$.
			\item For every $\Lang$-formula $\phi$ and $\epsilon > 0$, there exists an inf-formula $\psi$ such that $|\phi - \psi| < \epsilon$ (on the appropriate sort or domain) for all models of $\mathrm{T}$.
		\end{enumerate}
	\end{lem}
	
	The proof is similar to the quantifier elimination case, but more technical.  Since it has not been explicitly given in the literature for metric structures to our knowledge, we include the proof as an appendix.  The fact that quantifier elimination implies model completeness is immediate since Lemma \ref{lem: type QE} (1) implies Lemma \ref{lem: type QE} (3), or alternatively since Lemma \ref{lem: type QE} (2) implies Lemma \ref{lem: type MC} (2).
	
	\section{Quantifier elimination for tracial von Neumann algebras} \label{sec: QE}
	
	\subsection{Proof of Theorem \ref{thm: main QE}} \label{sec: QE proof}
	
	Toward the proof of Theorem \ref{thm: main QE}, first note that we can restrict our attention to type I algebras.  Indeed, the first author already showed that any tracial von Neumann algebra with a type II$_1$ summand does not admit quantifier elimination \cite[Theorem~1]{Farah2023} (another argument is given in Remark \ref{rem: alternate QE proof} below).  The next lemma will similarly allow us to eliminate summands of the form $M_n(\C) \otimes L^\infty[0,1]$ with $n \geq 2$, by showing that if either (1) or (2) in Theorem \ref{thm: main QE} happens, then there can be no such summands.
	
	\begin{lem} \label{lem: eliminate diffuse matrix term}
		Suppose that  $\cM$ is a tracial von Neumann algebra.  Assume either that $\Th(\cM)$ admits quantifier elimination or that any two projections of the same trace are conjugate by an automorphism of $\cM$. Then $\cM$ cannot have a direct summand of the form $M_n(\C) \otimes L^\infty[0,1]$ for $n \geq 2$.
	\end{lem}
	
	\begin{proof}
		By contrapositive, suppose that $\cM$ has a direct summand of the form $M_n(\C) \otimes L^\infty[0,1]$.  In $M_n(\C) \otimes L^\infty[0,1]$, consider the projections $p = 1 \otimes \mathbf{1}_{[0,1/n]}$ and $q = E_{1,1} \otimes 1$, where $E_{1,1}$ is the canonical matrix unit in $M_n(\C)$.  These two projections have the same trace, hence they have the same $*$-moments, i.e.\ the same quantifier-free type.  However, they do not have the same type because $p$ is central and $q$ is not central in $M_n(\C) \otimes L^\infty[0,1]$, hence also in $\cM$.  So $\cM$ cannot admit quantifier elimination.  Furthermore, since $p$ and $q$ do not have the same type, they cannot be conjugate by an automorphism of $\cM$.
	\end{proof}
	
	Therefore, it suffices to prove Theorem \ref{thm: main QE} in the case where $\cM$ is a direct sum of an optional $L^\infty[0,1]$ term and matrix algebras.  
	Let us decompose $\cM$ as follows:
	\[
	\cM = (L^\infty[0,1], \alpha_0) \oplus \left( \bigoplus_{j \in J} (M_{n_j}(\C), \alpha_j) \right).
	\]
	Here $\alpha_j$, for $j\in \{0\}\sqcup J$, are the weights of the direct summands. Thus $\alpha_0+\sum_{j\in J} n_j \alpha_j=1$. 
	
	We rely on the following classification of the automorphisms of $\cM$ (for background on the structure theory for finite-dimensional algebras, see e.g.\ \cite[\S 3.1]{Davidson1996}, \cite[\S 3.2]{JonesSunder1997}).  Every automorphism of $\cM$ is a composition of the following:
	\begin{enumerate}
		\item A direct sum of automorphisms of each component
		(a measure-space automorphism of $L^\infty[0,1]$ and a unitary conjugation of each $M_n(\C)$ term), 
		\item Swaps of matrix algebras $M_n(\C)$ of the same dimension and the same weight.
	\end{enumerate}
	
	We first focus on the atomic portion.
	
	\begin{lem} \label{lem: no identical matrix summands}
		Suppose that $\cM$ is a tracial von Neumann algebra such that any two projections of the same trace are conjugate by an automorphism of $\cM$. Then any two matrix summands of $\cM$ with a common dimension greater than or equal to 2 must have different weights.
	\end{lem}
	
	\begin{proof}
		Suppose there is some $j, k \in J$ so that $n_j = n_k \geq 2$, and $\alpha_j = \alpha_k$. 
		Let $p$ be a projection of rank $2$ in the $M_{n_j}(\C)$ summand, and let $q$ be a projection of rank $1$ in both the $M_{n_j}(\C)$ and $M_{n_k}(\C)$ summands (and $p$, $q$ are both 0 in all other summands.) Then $\tau(p) = \tau(q) = \frac{2 \alpha_j}{n_j}$, but $p$ and $q$ are not conjugate by any automorphism. 
	\end{proof}
	
	
	\begin{proof}[Proof of Theorem \ref{thm: main QE}]
		(1) $\implies$ (2).  Suppose that $\cM$ admits elimination of quantifiers. In order to deal with the diffuse $L^\infty$ term and the atomic terms separately, we first show that the central projection $1_{L^\infty}$ is a definable element (see \S \ref{sec: definable}).  Note that for each $k$, the set
		\[
		S_k = \{e_1,\dots,e_k \in P(\cM) \cap Z(\cM): e_i e_j = 0, \tau(e_j) = \alpha_0/k \text{ for } i, j = 1, \dots, k \}
		\]
		is definable using the definability of the center (see \cite[Lemma 4.2]{FHS2013}) and the stability of projections.  Moreover, if $x$ is any element satisfying
		\[
		\inf_{(e_1,\dots,e_k) \in S_k} d\left(x,\sum_{j=1}^k e_j \right) \leq \epsilon,
		\]
		then $x$ is $\epsilon$-close to a central projection that is divisible into $k$ central projections of trace $\alpha_0/k$.  If $k$ is large enough, then the sum of the weights of discrete summands that are less than or equal to $\alpha_0/k$ will be less than $\epsilon^2$.  Hence, $\sum_{j=1}^k e_j$ will be $2\epsilon$-close to $1_{L^\infty}$.  So $1_{L^\infty}$ is definable.
		
		Let $p, q$ be two projections with the same trace. As noted in the proof of Lemma \ref{lem: eliminate diffuse matrix term}, $p$ and $q$ then have the same quantifier-free type and hence they have the same type.  Because $1_{L^\infty}$ is definable, every formula over $L^\infty$ and every formula over $\cN := \cM \ominus L^\infty$ can be expressed as a definable predicate over $\cM$. Thus, $1_{L^\infty} p$ and $1_{L^\infty}q$ have the same type in $L^\infty[0,1]$ and $(1_\cN)p$ and $(1_\cN)q$ have the same type in $\cN$.  
		Then, $1_{L^\infty} p$ and $1_{L^\infty}q$ are two projections of the same trace in $L^{\infty}[0,1]$ and therefore conjugate by an automorphism.  Meanwhile, $(1 - 1_{L^\infty})p$ and $(1 - 1_{L^\infty})q$ have the same type in $\cN$, hence they are conjugate by an automorphism in some elementary extension $\tilde{\cN}$ of $\cN$.  Since $\cN$ is type I and atomic, $\tilde{\cN}$ must equal $\cN$ (see \cite[Proposition 4.3]{FG2023} or \cite[Proposition 3.7(2)]{Jekel2024Optimal}).  Thus, $(1_\cN)p$ and $(1_\cN)q$ are conjugate by an automorphism of $\cN$, and so $p$ and $q$ are conjugate by an automorphism of $\cM$.
		
		(2) $\implies$ (1)
		Let $\mathrm{T} := \Th(\cM)$. We must check that every $\mathrm{T}$ type is determined by its quantifier-free type.
		First note that all $\mathrm{T}$ types can be realized in $\cM$; indeed, $\cM^{\cU}$ is countably saturated (see \S \ref{sec: model theory prelims}) and is a direct sum of $L^\infty[0,1]^{\cU}$ and $(\C^{n_j})^{\cU} = \C^{n_j}$ and ${M_{n_j}(\C)^{\cU} = M_{n_j}(\C)}$.   Any tuple of elements in $L^\infty[0,1]^{\cU}$  has the same type as some tuple in $L^\infty[0,1]$, and swapping out the element in the $L^\infty[0,1]^{\cU}$ summand for one of the same type will not change the type of the overall element in $\cM^{\cU}$.
		
		Fix some $\mathbf{x} = (x_1,\dots,x_k)$ and $\mathbf{y} = (y_1,\dots,y_k)$ in $\cM$ with the same quantifier-free type.
		We shall build a sequence of automorphisms $\sigma_n$ of $\cM$ such that $\sigma_n(\mathbf{x}) \rightarrow \mathbf{y}$, so $\tp^\cM(\mathbf{x}) = \tp^\cM(\mathbf{y})$. 
		Since there are no identical matrix summands with the same weight by Lemma \ref{lem: no identical matrix summands}, the only possible automorphisms of $\cM$ are those which are a direct sum of automorphisms of each component, possibly composed with swaps of copies of $\C$ which have the same weight. This motivates the following decomposition of $\cM$, where we group together copies of $\C$ which have the same weight:\footnote{By \cite[Lemma~3.2]{FG2023}, the data used in \eqref{eq.alpha-j} is computable from the theory of $\cM$. For reader's convenience we provide a translation.  In the terminology of \cite{FG2023}, $\alpha_0=\rho_{\cM}(1,0)$, $\rho_{\cM}(m,0)=0$ for $m\geq 2$, $\rho_{\cM}(1,k)$, for $k\geq 1$, is the sequence in which each $\alpha_j$, for $j\in J_1$, appears $n_j$ times, arranged in decreasing order.  Finally, $\rho_{\cM}(n_j,1)=\alpha_j$ and $\rho_{\cM}(n,k)=0$ if $n\neq n_j$ for all $j$ or if $k\geq 2$.}
		\begin{equation}   \label{eq.alpha-j}
			\cM = (L^\infty[0,1], \alpha_0) \oplus \left( \bigoplus_{j \in J_1} (\C, \alpha_j)^{\oplus n_j} \right) \oplus \left( \bigoplus_{j \in J_2} (M_{n_j}(\C), \alpha_j) \right).
		\end{equation}
		We will build the automorphisms on each summand of \eqref{eq.alpha-j} separately.
		
		We start with the matrix summands.  Let $p_j$, $j \in J_2$, be the central projection onto the $j$th summand $M_{n_j}(\C)$, where $n_j \geq 2$.  We claim that $p_j\mathbf{x} = (p_j x_1,\dots,p_j x_k)$ and ${p_j \mathbf{y} = (p_j y_1,\dots,p_j y_k)}$ have the same quantifier-free type in $M_{n_j}(\C)$.  Let $f$ be a self-adjoint non-commutative $*$-polynomial.  For Borel $E \subseteq \R$, we have $\tau(1_E(f(\mathbf{x}))) = \tau(1_E(f(\mathbf{y})))$, so by assumption there is some automorphism $\sigma$ conjugating $1_E(f(\mathbf{x}))$ to $1_E(f(\mathbf{y}))$.  As noted above, the automorphism $\sigma$ must fix $p_j$, so $\sigma(p_j 1_E(f(\mathbf{x}))) = p_j 1_E(f(\mathbf{y}))$. Hence, $\tau(p_j 1_E(f(\mathbf{x}))) = \tau(p_j 1_E(f(\mathbf{y})))$, or equivalently ${\tr_{n_j}(1_E(f(p_j \mathbf{x}))) = \tr_{n_j}(1_E(f(p_j \mathbf{y})))}$. Since $E$ was arbitrary, $f(p_j \mathbf{x})$ and $f(p_j \mathbf{y})$ have the same empirical spectral distribution, hence also ${\tr_n(f(p_j \mathbf{x})) = \tr_n(f(p_j \mathbf{y}))}$.  This holds for all $f$, so the multivariate Specht's theorem \cite{Jing2015} implies that ${u_j p_j\mathbf{x}u_j^* = p_j \mathbf{y}}$ for some unitary $u \in M_{n_j}(\C)$.
		
		The same argument as in the matrix case shows that when $p_j$ for $j \in J_1$ is the central projection onto some summand of the form $\C^{n_j}$, $n_j \geq 1$, with each copy of $\C$ having the same weight $\alpha_j$, we obtain that $\qftp^{\C^{n_j}} (p_j \mathbf{x}) = \qftp^{\C^{n_j}}(p_j \mathbf{y})$, so some automorphism (i.e.\ permutation)  $\pi_j$ of $\C^{n_j}$ sends $p_j \mathbf{x}$ to $p_j \mathbf{y}$.
		
		Finally, let $p_0$ be the central projection onto the $L^\infty[0,1]$ summand. Then ${p_0 = 1 - \sum_{j \in J_1 \sqcup J_2} p_j}$, where $p_j$ is the central projection onto the $j$th summand of $\cM$. Hence, for any non-commutative $*$-polynomial $f$,
		\[
		\tau(p_0 f(\mathbf{x})) = \tau(f(\mathbf{x})) - \sum_{j \in J_1 \sqcup J_2} \tau(p_j f(\mathbf{x})) = \tau(f(\mathbf{y})) - \sum_{j \in J_1 \sqcup J_2} \tau(p_j f(\mathbf{y})) = \tau(p_0 f(\mathbf{y})),
		\]
		so we again obtain that $\qftp^{L^\infty}(p_0 \mathbf{x}) = \qftp^{L^\infty} (p_0 \mathbf{y})$.  By \cite[Lemma 2.16]{JekelModelEntropy}, there is a sequence of automorphisms $\alpha_n$ of $L^\infty[0,1]$ such that $\alpha_n(\mathbf{x}) \to \mathbf{y}$.
		
		To conclude, let $\sigma_n$ be the direct sum of the automorphisms in each summand of $\cM$ given by the arguments above, that is,
		\[
		\sigma_n = \alpha_n \oplus \bigoplus_{j \in J_1} \pi_j \oplus \bigoplus_{j \in J_2} \operatorname{Ad}_{u_j}.
		\]
		Then $\sigma_n(\mathbf{x}) \to \mathbf{y}$, so $\tp^{\cM}(\mathbf{x}) = \tp^{\cM}(\mathbf{y})$. Hence, $\cM$ admits elimination of quantifiers by Lemma \ref{lem: type QE}.
	\end{proof}
	
	\subsection{Tests for quantifier elimination} \label{sec: QE tests}
	
	The criterion for quantifier elimination of Theorem \ref{thm: main QE}, though simple, does not clearly indicate how to decide if a tracial von Neumann algebra admits quantifier elimination based on a given description as a direct sum of matrix algebras.  So we now give more explicit criteria, starting with the following characterization in terms of possible obstructions.
	
	\begin{prop} \label{prop: QE obstructions}
		A separable tracial von Neumann algebra $\cM$ admits quantifier elimination if and only if all the following conditions hold:
		\begin{enumerate}[(1)]
			\item $\cM$ is type I.
			\item $\cM$ has no summands of the form $M_n(\C) \otimes L^\infty[0,1]$ for $n \geq 2$.
			\item If $\cM$ has an $L^\infty[0,1]$ summand with weight $\alpha_0$, and if $p$ and $q$ are two projections in the atomic part, then either $\tau(p) = \tau(q)$ or $|\tau(p) - \tau(q)| > \alpha_0$.
			\item If $p$ and $q$ are two projections in the atomic part with $\tau(p) = \tau(q)$, then we have (letting $E_ {Z(\cM)}$ denote the center-valued trace in $\cM$) $E_{Z(\cM)}[p] = \sigma \circ E_{Z(\cM)}[q]$ where $\sigma$ is an automorphism of $\cM$ given by a permutation of one-dimensional summands with the same weight.
		\end{enumerate}
	\end{prop}
	
	\begin{proof}
		Suppose $\cM$ admits quantifier elimination.  Then \cite[Theorem~1]{Farah2023} implies (1) and Lemma \ref{lem: eliminate diffuse matrix term} implies (2).
		
		For (3), suppose for contradiction that there are two projections $p$ and $q$ in the atomic part with $0 < |\tau(p) - \tau(q)| \leq \alpha_0$, and without loss of generality suppose that $\tau(p) < \tau(q)$.  Let $p'$ be a projection in $L^\infty[0,1]$ such that $\tau(p') = \tau(q) - \tau(p)$.  Then $q$ and $p' + p$ have the same trace but are not equivalent by an automorphism, so by Theorem \ref{thm: main QE}, $\cM$ does not have quantifier elimination.
		
		For (4), let $p$ and $q$ be projections in the atomic part with $\tau(p) = \tau(q)$.  By Theorem \ref{thm: main QE}, $p$ and $q$ are conjugate by an automorphism.  Hence also $E_{Z(\cM)}[p]$ and $E_{Z(\cM)}[q]$ are conjugate by an automorphism.  In light of Lemma \ref{lem: no identical matrix summands}, every automorphism must fix the central projections associated to $M_n(\C)$ terms for $n \geq 2$.  Thus, $E_{Z(\cM)}[p]$ and $E_{Z(\cM)}[q]$ must have equal components in each of the $M_n(\C)$ summands for $n \geq 2$.  So they differ by an automorphism that merely permutes the one-dimensional summands.
		
		Conversely, assume (1)--(4).  Let $p$ and $q$ be two projections of the same trace.  Using (3), the traces of $p$ and $q$ in the $L^\infty[0,1]$ summand must agree, so there is an automorphism of $\cM$ such that $\alpha(p) - q$ is in the atomic part of $\cM$.  So assume without loss of generality that $p$ and $q$ are in the atomic part.  By (4), after applying an automorphism, we can assume that $E_{Z(\cM)}[p] = E_{Z(\cM)}[q]$.  Hence, the components of $p$ and $q$ in each direct summand $M_n(\C)$ of $\cM$ (where $n \geq 1$), have the same rank, and hence are unitarily conjugate. Overall, $p$ and $q$ are conjugate by an automorphism.  By Theorem \ref{thm: main QE}, $\cM$ admits quantifier elimination.
	\end{proof}
	
	Next, we describe how to test condition (4) for the atomic part in terms of the weights in the direct sum decomposition.  As motivation, recall that by Lemma \ref{lem: no identical matrix summands}, two matrix algebras of the same dimension cannot have the same weight.  In fact, there are many more constraints of a similar nature.  For instance, if
	\[
	\cM = (\C,1/2) \oplus (\C,1/3) \oplus (\C,1/6),
	\]
	then $1 \oplus 0 \oplus 0$ and $0 \oplus 1 \oplus 1$ have the same trace but are not automorphically conjugate.  Another example is if
	\[
	\cM = (\C,2/5) \oplus (M_3(\C),3/5),
	\]
	then $\cM$ does not admit quantifier elimination since a rank $2$ projection in the second summand has the same trace as $1$ in the first summand.  Hence, we must consider various ways that zero could be written as a linear combinations of ranks of projections from different summands.  More generally, as in Proposition \ref{prop: QE obstructions} (3), quantifier elimination requires that no number smaller than $\alpha_0$ can be written as such a linear combination.  This gives essentially all the conditions that are needed, though one must also handle the one-dimensional summands carefully since Lemma \ref{lem: no identical matrix summands} only applies for $n \geq 2$.
	
	\begin{prop} \label{prop: QE linear inequality}
		Let $\cM$ be a separable tracial von Neumann algebra.  Then $\cM$ admits quantifier elimination if and only if $\cM$ has a decomposition of the form:
		\[
		\cM = (L^\infty[0,1], \alpha_0) \oplus \left( \bigoplus_{j \in J_1} (\C, \alpha_j)^{\oplus n_j} \right) \oplus \left( \bigoplus_{j \in J_2} (M_{n_j}(\C), \alpha_j) \right),
		\]
		where for some countable sets $J_1, J_2$, such that
		\begin{enumerate}[(1)]
			\item The weights satisfy $\alpha_0 \geq 0$ and $\alpha_j > 0$ for $j \in J_1 \cup J_2$, and the weights sum to $1$.
			\item The indices $\alpha_j$ for $j \in J_1$ are distinct, that is, we have grouped together \emph{all} one-dimensional summands of the same weight in our decomposition.
			\item For all choices of integers $|r_j| \leq n_j$ for $j \in J_1 \cup J_2$ which are not all zero, we have
			\[
			\left| \sum_{j \in J_1} r_j \alpha_j + \sum_{j \in J_2} \frac{r_j \alpha_j}{n_j} \right| > \alpha_0.
			\]
		\end{enumerate}
	\end{prop}
	
	\begin{proof}
		Suppose $\cM$ admits quantifier elimination.  We already know $\cM$ decomposes into an optional $L^\infty[0,1]$ term and an atomic part.  By grouping the one-dimensional terms with the same weight, we obtain a direct sum decomposition satisfying conditions (1) and (2).  It remains to check condition (3).  By contrapositive, suppose that there exist integers $|r_j| \leq n_j$ satisfying
		\[
		\left| \sum_{j \in J_1} r_j \alpha_j + \sum_{j \in J_2} \frac{r_j \alpha_j}{n_j} \right| \leq \alpha_0.
		\]
		For $j \in J_1$, let $p_j$ and $q_j$ be projections in $(\C,\alpha_j)^{\oplus n_j}$ such that
		\[
		\rank(p_j) = \max(r_j,0), \qquad \rank(q_j) = \max(-r_j,0).
		\]
		Similarly, for $j \in J_2$, let $p_j$ and $q_j$ be projections in $(M_{n_j}(\C),\alpha_j)$ with the same rank conditions.  Thus, $\rank(p_j) - \rank(q_j) = r_j$.  Finally, let
		\[
		t = \sum_{j \in J_1} r_j \alpha_j + \sum_{j \in J_2} \frac{r_j \alpha_j}{n_j},
		\]
		and let $p_0$ and $q_0$ be projections in $(L^\infty[0,1],\alpha_0)$ such that $\tau(p_0) = \max(-t,0)$ and $\tau(q_0) = \max(t,0)$, so that $\tau(p_0) - \tau(q_0) = -t$.  Let
		\[
		p = p_0 \oplus \bigoplus_{j \in J_1} p_j \oplus \bigoplus_{j \in J_2} p_j, \qquad q = q_0 \oplus \bigoplus_{j \in J_1} q_j \oplus \bigoplus_{j \in J_2} q_j.
		\]
		By construction,
		\[
		\tau(p) - \tau(q) = \tau(p_0) - \tau(q_0) + \sum_{j \in J_1} \alpha_j r_j + \sum_{j \in J_2} \frac{\alpha_j r_j}{n_j} = 0.
		\]
		However, $p$ and $q$ are not automorphically conjugate.  Indeed, $r_j$ is nonzero for some $j$.  If $j \in J_1$, the components of $p$ and $q$ in the central summand $(\C,\alpha_j)^{\oplus n_j}$ have different ranks, and $(\C,\alpha_j)^{\oplus n_j}$ is invariant under automorphisms because we grouped together all the terms with the same weight.  Similarly, if $j \in J_2$, then the components of $p$ and $q$ in $(M_{n_j}(\C),\alpha_j)$ have different ranks, and by Lemma \ref{lem: no identical matrix summands}, $(M_{n_j}(\C),\alpha_j)$ must be invariant under automorphisms since there is only one summand with a given dimension and weight.  Hence, if (3) does not hold, then $\Th(\cM)$ does not admit quantifier elimination.
		
		Conversely, suppose $\cM$ has a decomposition satisfying (1) - (3).  Consider two projections $p = p_0 \oplus \bigoplus_{j \in J_1 \sqcup J_2} p_j$ and $q= q_0 \oplus \bigoplus_{j \in J_1 \sqcup J_2} q_j$ in $\cM$ with the same trace. Then 
		\[
		\tau(p_0) - \tau(q_0) = \sum_{j \in J_1 \sqcup J_2} \frac{\alpha_j (\text{rank}(q_j) - \text{rank}(p_j))}{n_j}.
		\]
		Hence,
		\[
		\left| \sum_{j \in J_1} \alpha_j (\operatorname{rank}(q_j) - \operatorname{rank}(p_j)) + \sum_{j \in J_2} \frac{\alpha_j (\operatorname{rank}(q_j) - \operatorname{rank}(p_j))}{n_j} \right| = |\tau(p_0) - \tau(q_0)| \leq \alpha_0.
		\]
		By condition (3), this forces $\rank(p_j) = \rank(q_j)$ for all $j \in J_1 \sqcup J_2$.  In particular, for $j \in J_1$, $p_j$ and $q_j$ are projections in $(\C,\alpha_j)^{\oplus n_j}$ with the same rank and hence conjugate by an automorphism permuting the summands.  Moreover, for $j \in J_2$, $p_j$ and $q_j$ are projections in $M_{n_j}(\C)$ with the same rank, hence they are unitarily conjugate.  Finally, since $p_j$ and $q_j$ have the same trace for $j \in J_1 \sqcup J_2$, we deduce that $p_0$ and $q_0$ have the same trace in $L^\infty[0,1]$ and hence they are conjugate by a measure-preserving transformation.  Patching the automorphisms on each summand together, $p$ and $q$ are automorphically conjugate.  Thus, by Theorem \ref{thm: main QE}, $\cM$ has quantifier elimination.
	\end{proof}

	\section{Model completeness for $\mathrm{II}_1$ factors} \label{sec: MC factor proof}
	
	This section proves Theorem \ref{thm: main MC} in the case of a $\mathrm{II}_1$ factor $\cM$.  The proof is a more sophisticated variant of \cite[Lemma 2.1]{Farah2023}, which was in turn based on \cite[Corollary~6.11]{Bro:Topological}.
	
	Our construction is based on random matrix theory.  Let $\mathbb{U}_n$ denote the unitary group of $M_n(\mathbb C)$.  As a compact Lie group, $\mathbb{U}_n$ has a unique left-invariant probability measure, called the \emph{Haar measure}.  By a \emph{Haar random unitary}, we mean a $\mathbb{U}_n$-valued random variable $U^{(n)}$ whose probability distribution is the Haar measure, i.e., $\mathbb{E}[f(U^{(n)})] = \int_{\mathbb{U}_n} f(u)\,d\operatorname{Haar}(u)$ for every continuous function $f$ on $\mathbb{U}_n$.  Let $U_1^{(n)}$, $U_2^{(n)}$, $U_3^{(n)}$, and $U_4^{(n)}$ be independent Haar random unitaries.  We assume throughout that they are on the same probability space $(\Omega,\mathcal{F},P)$.
	
	Consider the tensor decomposition $\cM \cong M_n(\C) \otimes \cM^{1/n}$, where $\cM^{1/n}$ is the $1/n$ compression of $\cM$ \cite[\S 2.6-2.8]{MvNROO4}; for each $n$, we fix a decomposition for the entire argument, and write $\cM = M_n(\C) \otimes \cM^{1/n}$.  We set
	\[
	\mathbf{X}^{(n)} = (X_1^{(n)},X_2^{(n)},X_3^{(n)}) = (U_1^{(n)} \otimes 1_{\cM^{1/n}},U_2^{(n)} \otimes 1_{\cM^{1/n}}, U_3^{(n)} \otimes 1_{\cM^{1/n}})
	\]
	and
	\[
	\mathbf{Y}^{(n)} = (Y_1^{(n)},Y_2^{(n)},Y_3^{(n)}) = ((U_1^{(n)} \oplus U_1^{(n)}) \otimes 1_{\cM^{1/2n}}, (U_2^{(n)} \oplus U_2^{(n)}) \otimes 1_{\cM^{1/2n}}, (U_3^{(n)} \oplus U_4^{(n)}) \otimes 1_{\cM^{1/2n}}).
	\]
	Fix a free ultrafilter $\cU$ on $\cN$ and consider $\mathbf{X}(\omega) = [\mathbf{X}^{(n)}(\omega)]_{n \in \cN}$ and $\mathbf{Y}(\omega) = [\mathbf{Y}^{(n)}(\omega)]_{n \in \N}$.  Thus, $\mathbf{X}$ and $\mathbf{Y}$ are intuitively tuples of random elements of $\cM^{\cU}$; however, we have to proceed carefully because $\mathbf{X}$ and $\mathbf{Y}$ are not necessarily measurable functions of $\omega$ (see \cite[\S 6]{GaoJekelIntegral}).  Thus, formally, our arguments are based on first fixing an outcome $\omega$ for which the $\mathbf{X}^{(n)}$'s satisfy some conditions, and then using the values of $\mathbf{X}$ and $\mathbf{Y}$ associated to this $\omega$.

	\subsection{Outline of the proof} \label{sec: outline}
	The outline of the argument is as follows:
	\begin{enumerate}[(1)]
		\item \label{item1} Almost surely, for every $\inf$-formula $\phi$, $\phi^{\cM^\cU}(\mathbf{Y}) \leq \phi^{\cM^{\cU}}(\mathbf{X})$.
		\item Almost surely, the commutant $\mathbf{X}' \cap \cM^{\cU}$ is given by
		\[
		\cA = \prod_{n \to \cU} (\C 1_{M_n(\C)} \otimes \cM^{1/n}) \subseteq \prod_{n \to \cU} (M_n(\C) \otimes \cM^{1/n}).
		\]
		\item Almost surely, the commutant $\mathbf{Y}' \cap \cM^{\cU}$ is given by
		\[
		\cB = \prod_{n \to \cU} [(\C 1_{M_n(\C)} \oplus \C 1_{M_n(\C)}) \otimes \cM^{1/2n}] \subseteq \prod_{n \to \cU} (M_{2n}(\C) \otimes \cM^{1/2n}).
		\]
		\item Consequently, $\mathbf{X}' \cap \cM^{\cU}$ has trivial center but $\mathbf{Y}' \cap \cM^{\cU}$ does not and so $\mathbf{X}$ and $\mathbf{Y}$ do not have the same type.
		\item By Lemma \ref{lem: type MC} together with (1) and (4), $\Th(\cM)$ is not model complete.
	\end{enumerate}
	The notation explained above will be fixed throughout the section.  Moreover, we continue with the standing assumption that $M_2(\cM)$ embeds into $\cM^{\cU}$ for some ultrafilter $\cU$, but this will only be used in the proof of \eqref{item1}, in  Lemma \ref{lem: inf formula limit}.
	
	\subsection{Concentration of measure and approximate embedding}
	
	For step (1), we use the following concentration of measure estimate which is based on the log-Sobolev inequality of Gross \cite{Gross1975}.  The application of concentration in random matrix theory is due to Ben Arous and Guionnet \cite{BAG1997}; see also \cite{Guionnet2009} and \cite[\S 2.3 and 4.4]{AGZ2009}.
	
	\begin{prop}[{See \cite[\S 4.4 and Appendix F.6]{AGZ2009} and \cite[Theorem 5.16-5.17]{Meckes2019}}] \label{prop: concentration}
		Let ${f: \mathbb{U}_n^{\times m} \to \R}$ be an $L$-Lipschitz function with respect to $\norm{\cdot}_2$.  Let $\mathbf{U}^{(n)}$ be a random element of $\mathbb{U}_n^{\times m}$ with probability distribution given by the Haar measure.  Then for some positive constant $c$ independent of $n$, for all $\delta > 0$,
		\[
		\mathbb{P}( |f(\mathbf{U}^{(n)}) - \mathbb{E}[f(\mathbf{U}^{(n)})]| \geq \delta) \leq e^{-cn^2 \delta / L^2}.
		\]
	\end{prop}
	
	\begin{lem} \label{lem: type convergence}
		For every $3$-variable formula $\phi$, there is a constant $C(\varphi)$ such that
		\begin{equation} \label{eq: type convergence}
			\lim_{n \to \cU} \phi^{\cM}(\mathbf{X}^{(n)}) = C(\varphi) \text{ for a.e. } \omega \in \Omega.
		\end{equation}
		In particular, $\lim_{n \to \cU} \tp^{\cM}(\mathbf{X}^{(n)})$ is almost surely constant.
	\end{lem}
	
	\begin{proof}
		To prove the claims, it suffices to show \eqref{eq: type convergence} holds almost surely for each $\phi$ in a countable dense set of formulas (as usual in measure theory, ``almost surely'' distributes over countable conjunctions).
		
		In fact, the dense set of formulas can be chosen to be Lipschitz.  Indeed, a formula will be Lipschitz as long as the atomic formulas and the connectives used are all Lipschitz; the quantifiers do not cause any issue since the supremum of a family of $L$-Lipschitz functions is $L$-Lipschitz.  The atomic formulas are traces of non-commutative polynomials, and for every non-commutative polynomial $p$ and $R > 0$, there is some $L$ such that $\tau(p)$ is $L$-Lipschitz with respect to $\norm{\cdot}_2$ on each operator norm ball of radius $R$.  The connectives in the language are continuous functions $\R^m \to \R$, which can all be approximated on compact sets by Lipschitz functions.
		
		So assume that $\phi$ is an $L$-Lipschitz formula in three variables. Note that $\mathbf{X}^{(n)}$ depends in a Lipschitz manner upon $\mathbf{U}^{(n)} = (U_1^{(n)},U_2^{(n)},U_3^{(n)})$; indeed, the mapping $M_n(\C) \to \cM$ given by $u \mapsto u \otimes 1_{\cM^{1/n}}$ is $1$-Lipschitz.  In particular, $\phi^{\cM}(\mathbf{X}^{(n)})$ is an $L$-Lipschitz function of $\mathbf{U}^{(n)}$.  Therefore, applying Proposition \ref{prop: concentration} with $\delta = 1/n$,
		\[
		\mathbb{P}( |\phi^{\cM}(\mathbf{X}^{(n)}) - \mathbb{E}[\phi^{\cM}(\mathbf{X}^{(n)})])| \geq 1/n) \leq e^{-cn / L^2}.
		\]
		By the Borel-Cantelli lemma, this implies that almost surely
		\[
		\lim_{n \to \infty} |\phi^{\cM}(\mathbf{X}^{(n)}) - \mathbb{E}[\phi^{\cM}(\mathbf{X}^{(n)})]| = 0, \quad \text{hence} \quad 
		\lim_{n \to \cU} \phi^{\cM}(\mathbf{X}^{(n)}) = \lim_{n \to \cU} \mathbb{E}[\phi^{\cM}(\mathbf{X}^{(n)})].\qedhere
		\]
	\end{proof}
	
	\begin{lem} \label{lem: inf formula limit}
		Almost surely, for every $\inf$-formula $\phi$ in three variables,
		\begin{equation} \label{eq: inf formula inequality}
			\lim_{n \to \cU} \phi^{\cM}(\mathbf{Y}^{(n)}) \leq \lim_{n \to \cU} \phi^{\cM}(\mathbf{X}^{(n)}).
		\end{equation}
	\end{lem}
	
	\begin{proof}
		Let $\tilde{\mathbf{X}}^{(n)}$ be defined analogously to $\mathbf{X}^{(n)}$ but with $U_4^{(n)}$ in place of $U_3^{(n)}$, that is, $\tilde{\mathbf{X}}^{(n)} = (U_1^{(n)} \otimes 1_{\cM^{1/n}}, U_2^{(n)} \otimes 1_{\cM^{1/n}}, U_4^{(n)} \otimes 1_{\cM^{1/n}})$.  Since $\tilde{\mathbf{X}}^{(n)}$ has the same probability distribution as $\mathbf{X}^{(n)}$, the almost sure limit of $\tp^{\cM}(\tilde{\mathbf{X}}^{(n)})$ agrees with that of $\mathbf{X}^{(n)}$.
		
		In the following, we fix an outcome $\omega$ in the probability space such that the limit as $n \to \cU$ of the type of $\mathbf{X}^{(n)}$ and the type of $\tilde{\mathbf{X}}^{(n)}$ at $\omega$ agree with the almost sure limits given by Lemma \ref{lem: type convergence}.  Let $\phi$ be an existential formula.  Then $\phi$ can be expressed as
		\[
		\phi(x_1,x_2, x_3) = \inf_{z_1, \dots, z_k} \psi(x_1,x_2, x_3, z_1,\dots,z_k),
		\]
		where $\psi$ is a quantifier-free formula and each $z_j$ ranges over the unit ball.  Since $\cM^{\cU}$ is countably saturated (see \S \ref{sec: model theory prelims}), there exists some $\mathbf{Z} \in (\cM^{\cU})_1^k$ such that $\phi^{\cM^{\cU}}(\mathbf{X}) = \psi^{\cM^{\cU}}(\mathbf{X},\mathbf{Z})$.  Now because $\mathbf{X}$ and $\tilde{\mathbf{X}}$ have the same type in $\cM^{\cU}$, there also exists some $\tilde{\mathbf{Z}} \in (\cM^{\cU})_1^k$ such that $(\mathbf{X},\mathbf{Z})$ and $(\tilde{\mathbf{X}},\tilde{\mathbf{Z}})$ have the same quantifier-free type.
		
		In the hypotheses of Theorem \ref{thm: main MC}, we assumed there is an embedding $i: M_2(\cM) \to \cM^{\cV}$ for some ultrafilter $\cV$.\footnote{By standard methods, one can choose $\cV=\cU$ (see \cite[Theorem 16.7.4]{Fa:STCstar}), but this is besides the point.} 
		Let $i^{(n)}$ be the corresponding embedding
		\[
		i^{(n)}: \cM^{1/n} = M_2(\cM)^{1/2n} \to (\cM^{\cV})^{1/2n} \cong (\cM^{1/2n})^{\cV}.
		\]
		Then let
		\[
		i^{\cU} = \prod_{n \to \cU} (\id_{M_n(\C)} \otimes i^{(n)}): \cM^{\cU} = \prod_{n \to \cU} (M_n(\C) \otimes \cM^{1/n}) \to \prod_{n \to \cU} (M_n(\C) \otimes (\cM^{1/2n})^{\cV}) \cong ((\cM^{1/2})^{\cV})^{\cU}.
		\]
		Consider $i^{\cU}(\mathbf{X}) \oplus i^{\cU}(\tilde{\mathbf{X}})$ and $i^{\cU}(\mathbf{Z}) \oplus i^{\cU}(\tilde{\mathbf{Z}})$ as elements of
		\[
		M_2(((\cM)^{1/2})^{\cV})^{\cU}) = (\cM^{\cV})^{\cU} = (\cM^{\cU})^{\cV}.
		\]
		Note that $(i^{\cU}(\mathbf{X}) \oplus i^{\cU}(\tilde{\mathbf{X}}), \, i^{\cU}(\mathbf{Z}) \oplus i^{\cU}(\tilde{\mathbf{Z}}))$ has the same quantifier-free type as $(\mathbf{X},\mathbf{Z})$, and in particular,
		\[
		\phi^{(\cM^{\cU})^{\cV}}(i^{\cU}(\mathbf{X}) \oplus i^{\cU}(\tilde{\mathbf{X}})) \leq \psi^{(\cM^{\cU})^{\cV}}(i^{\cU}(\mathbf{X}) \oplus i^{\cU}(\tilde{\mathbf{X}}), \, i^{\cU}(\mathbf{Z}) \oplus i^{\cU}(\tilde{\mathbf{Z}})) = \phi^{\cM^{\cU}}(\mathbf{X}).
		\]
		On the other hand,
		\[
		i^{\cU}(\mathbf{X}) \oplus i^{\cU}(\tilde{\mathbf{X}}) = j(\mathbf{Y}),
		\]
		where $j$ is the diagonal embedding
		\[
		j: \cM^{\cU} \to (\cM^{\cU})^{\cV} \text{ or equivalently } \prod_{n \to \cU} (M_{2n}(\C) \otimes \cM^{1/2n}) \to \prod_{n \to \cU} (M_{2n}(\C) \otimes (\cM^{1/2n})^{\cV}).
		\]
		Hence,
		\[
		\phi^{\cM^{\cU}}(\mathbf{Y}) = \phi^{(\cM^{\cU})^{\cV}}(j(\mathbf{Y})) = \phi^{(\cM^{\cU})^{\cV}}(i^{\cU}(\mathbf{X}) \oplus i^{\cU}(\tilde{\mathbf{X}})) \leq \phi^{\cM^{\cU}}(\mathbf{X}).
		\]
		This proves the asserted inequality \eqref{eq: inf formula inequality}.
	\end{proof}

	\subsection{Spectral gap and quantum expanders}
	
	For steps (2) and (3) from \S \ref{sec: outline}, we want precise control over the commutants of the $\mathbf{X}$ and $\mathbf{Y}$.  Hence, we will use the notion of \emph{spectral gap} for an inclusion $\cN \subseteq \cM$ of tracial von Neumann algebras.  For $d \in \N$ and $C > 0$, we say that $\cN \subseteq \cM$ has \emph{$(C,d)$-spectral gap} if there exist $x_1$, \dots, $x_d$ in the unit ball $B_1^{\cN}$ such that
	\begin{equation} \label{eqn: spectral gap def}
		d(y,\cN' \cap \cM)^2 \leq C \sum_{j=1}^d \norm{[x_j,y]}_2^2 \text{ for } y \in \cM,
	\end{equation}
	where $\cN' \cap \cM = \{z \in \cM: [z,x] = 0 \text{ for } x \in \cN\}$.  If this is true for some $d$ and $C$, we say that $\cN \subseteq \cM$ has \emph{spectral gap}.  In the case $\cN = \cM$, note that $\cN' \cap \cM$ reduces to the center $Z(\cM)$, and in this case, we will say simply that $\cM$ has spectral gap.  The relevance of spectral gap for continuous logic was already observed by Goldbring \cite{Goldbring2023spectralgap}, who showed that spectral gap for $\cN \subseteq \cM$ implies that $\cN' \cap \cM$ is a definable set with parameters from $\cN$.
	
	It is well known that when the $x_j$'s in \eqref{eqn: spectral gap def} are unitaries, the inequality can be reformulated in the following way, which will motivate our use of quantum expanders.
	
	\begin{lem} \label{lem: spectral gap}
		Let $\cN \subseteq \cM$ be an inclusion of tracial von Neumann algebras and $\epsilon > 0$, let $u_1$, \dots, $u_d$ be unitaries in $\cN$.  Then the following are equivalent:
		\begin{enumerate}[(1)]
			\item For $a \in \cM$,
			\[
			\norm{a - E_{\cN' \cap \cM}(a)}_2^2 \leq \frac{1}{\epsilon} \sum_{j=1}^d \norm{[u_j,a]}_2^2.
			\]
			\item For $a \in \cM$,
			\[
			\norm{\sum_{j=1}^d u_j(a - E_{\cN' \cap \cM}(a)) u_j^* + u_j^*(a - E_{\cN}(a))u_j}_2 \leq (2d - \epsilon) \norm{a - E_{\cN' \cap \cM}(a)}_2.
			\]
		\end{enumerate}
	\end{lem}
	
	\begin{proof}
		Let $T: L^2(\cM) \to L^2(\cM)^d$ be given by $T(a) = ([u_1,a], \dots, [u_d,a])$.  By elementary computation,
		\[
		T^*T(a) = 2d\,a - \sum_{j=1}^d u_jau_j^* - \sum_{j=1}^d u_j^*au_j.  
		\]
		Let $\mathcal{P} = \cN' \cap \cM$.  Note that $T$ vanishes on $\mathcal{P}$ and $a - E_{\mathcal{P}}(a)$ is the orthogonal projection of $a$ onto $\mathcal{P}^{\perp}$.  Therefore, condition (1) can be restated as $\epsilon \norm{a}_2^2 \leq \norm{T(a)}_2^2 = \scal{a,T^*T(a)}$ for $a \in \mathcal{P}^{\perp}$,
		which is equivalent to the spectrum of $T^*T |_{\mathcal{P}^{\perp}}$ being contained in $[\epsilon,\infty)$.  Meanwhile, condition (2) can be restated as $\norm{(2d - T^*T)|_{\mathcal{P}^{\perp}}} \leq 2d - \epsilon$; since $\norm{T^*T} \leq 2d$, this is equivalent to the above.
	\end{proof}
	
	\emph{Quantum expanders} are defined as follows.  For $\epsilon > 0$ and $d \geq 2$, a \emph{$(d,\epsilon)$-quantum expander} is a sequence of $d$-tuples of $n \times n$ unitaries $U_1^{(n)}$, \dots, $U_d^{(n)}$ such that for $A \in M_n(\C)$,
	\begin{equation} \label{eq.4.10}
		\norm{\sum_{j=1}^d U_j^{(n)} (A - \tr_n(A)) (U_j^{(n)})^*}_2 \leq (d - \epsilon) \norm{A - \tr_n(A)}_2.
	\end{equation}
	This estimate has the same form as Lemma \ref{lem: spectral gap} except that the latter is symmetrized with respect to $u_j$ and $u_j^*$.  We remark that ${(U_1^{(n)}, \dots, U_d^{(n)}, (U_1^{(n)})^*, \dots, (U_d^{(n)})^*)}$ is a $(2d,2\epsilon)$-quantum expander whenever $(U_1^{(n)}, \dots, U_d^{(n)})$ is a $(d,\epsilon)$-quantum expander; this follows because the adjoint of the map $A \mapsto \sum_{j=1}^d U_j^{(n)} A (U_j^{(n)})^*$ is the map $A \mapsto \sum_{j=1}^d (U_j^{(n)})^* A U_j^{(n)}$.
	
	The following relationship between spectral gap and quantum expanders is immediate from applying Lemma \ref{lem: spectral gap} with $\cN = \cM = M_n(\C)$ and $\cN' \cap \cM = \C 1$.
	
	\begin{cor} \label{cor: quantum expanders}
		Unitaries $U_1^{(n)}$, \dots, $U_d^{(n)}$ witness $(d,1/\epsilon)$ spectral gap for $M_n(\C)$ if and only if $(U_1^{(n)},\dots,U_d^{(n)},(U_1^{(n)})^*,\dots,(U_d^{(n)})^*)$ is a $(2d,\epsilon)$-quantum expander.
	\end{cor}
	
	Our argument uses Hastings's result that random unitaries give quantum expanders with high probability \cite{Hastings2007}; a similar result with matrix amplifications was shown by Pisier \cite{PisierExpanders}, and a generalization to other unitary representations was proved in \cite{BC2022}.   We remark as well that various other constructions of quantum expanders could have been used instead.  (A rich variety of determistic constructions exists, for instance, based on discrete Fourier transforms on non-abelian groups \cite{AS2004,BT2007,BST2010}, quantum versions of Margulis expanders \cite{GE2008}, systematic adaptation of classical expanders \cite{Harrow2008}, and zig-zag constructions \cite[\S 4]{BST2010}.)  Moreover, if $G$ is a group with property (T) (see \cite{BdlHVpropertyT} for background) with generators $g_1$, \dots, $g_d$, and $(\pi_j)_{j \in \N}$ is a sequence of irreducible unitary representations of $G$ on $\C^{n_j}$, then $(\pi_j(g_1),\dots,\pi_j(g_d))$ is a $(d,\epsilon)$-quantum expander where $\epsilon$ is related to the Kazhdan constant; thus, for instance, one can obtain quantum expanders from irreducible representations of $G = SL_3(\Z)$.  Property (T) groups and quantum expanders can be applied in many of the same contexts; see for instance the two proofs of \cite[Lemma 4.3]{ioana2021almost}.
	
	For the reader's convenience, we recall the precise statment of Hastings' result.
	
	\begin{thm}[{Hastings \cite{Hastings2007}, see also \cite[Lemma 1.8]{PisierExpanders}}] \label{thm: Hastings}
		Let $U_1^{(n)}$, \dots, $U_d^{(n)}$ be independent Haar random unitary matrices, and consider the (random) map $\Phi^{(n)}: M_n(\C) \to M_n(\C)$,
		\[
		\Phi^{(n)}(A) = \frac{1}{2d} \sum_{j=1}^d (U_j^{(n)} A (U_j^{(n)})^* + (U_j^{(n)})^*AU_j^{(n)})
		\]
		Let $\lambda_1^{(n)} \geq \lambda_2^{(n)} \geq \dots$ be the eigenvalues of $\Phi^{(n)}$ (here $\lambda_1^{(n)} = 1$ with eigenspace the span of the identity matrix).   Then almost surely
		\[
		\lim_{n \to \infty} \lambda^{(n)}_2 = \frac{\sqrt{2d - 1}}{d}.
		\]
	\end{thm}
	
	\begin{proof}
		The situation above is the Hermitian case with $D = 2d$ in Hastings's terminology.  Hastings \cite{Hastings2007} at the top of the second page asserts convergence in probability of $\lambda_2^{(n)}$.  Hastings's arguments in fact yield almost sure convergence.  Indeed, $\liminf_{n \to \infty} \lambda_2^{(n)} \geq \sqrt{2d-1} / d$ follows from a deterministic lower bound on $\lambda_2^{(n)}$ in \cite[eq (12)]{Hastings2007}  which gives (using $\lambda_H=2\sqrt{D-1}/D=\sqrt{2d-1}/d$, see \cite[(3)]{Hastings2007}),  $\lambda_2 \geq \frac{2d-1}d(1-O(\ln(\ln(n))/\ln(n))$.
		
		For the converse inequality,  at the end of \S II.F, Hastings shows that for $c >1$, the probability that $\lambda_2^{(n)}$ is greater than $c \sqrt{2d - 1}/ d$ is bounded by
		\[
		c^{-(1/4)n^{2/15}} (1 + O(\log(n)n^{-2/15})).
		\]
		Because this is summable, the Borel-Cantelli lemma implies that almost surely we have $\limsup_{n \to \infty} \lambda_2^{(n)} \leq c \lambda_H$.  Since $c > 1$ was arbitrary, this yields almost sure convergence.
	\end{proof}

	\begin{cor} \label{cor: Hastings spectral gap}
		Let $d > 1$, and let $U_1^{(n)}$, \dots, $U_d^{(n)}$ be Haar random unitary matrices. Then almost surely, for sufficiently large $n$, we have for all $A \in M_n(\C)$,
		\begin{equation} \label{eq: matrix spectral gap}
			\norm{A - \tr_n(A)}_2^2 < \frac{d}{(d-1)^2} \sum_{j=1}^d \norm{[A,U_j^{(n)}]}_2^2.
		\end{equation} 
	\end{cor}
	
	\begin{proof}
		Let $\Phi^{(n)}$ be as in the previous theorem.  Note that $\ker(\tr_n)$ is the orthogonal complement of $\C 1$, which is the $\lambda_1^{(n)}$-eigenspace of $\Phi^{(n)}$.  Hence,
		\[
		\norm{ \sum_{j=1}^d (U_j^{(n)} (A - \tr_n(A)) (U_j^{(n)})^* + (U_j^{(n)})^* (A - \tr_n(A)) U_j^{(n)}) }_2 \leq 2d \lambda_2^{(n)} \norm{A - \tr_n(A)}_2.
		\]
		This means that $U_1^{(n)}$, \dots, $U_d^{(n)}$ satisfy Lemma \ref{lem: spectral gap} (2) with $\cN = \mathbb{C}$, $\cM = M_n(\mathbb{C})$, and $2d - \epsilon^{(n)} = 2d \lambda_2^{(n)}$.  Hence, by Lemma \ref{lem: spectral gap}, the $U_j^{(n)}$ witness spectral gap for $\mathbb{C} \subseteq M_n(\mathbb{C})$ with constant $\epsilon^{(n)} = 2d(1 - \lambda_2^{(n)})$.  By Hastings's theorem, almost surely,
		\[
		\frac{1}{\epsilon^{(n)}} = \frac{1}{2d(1 - \lambda_2^{(n)})} \to \frac{1}{2d - 2 \sqrt{2d-1}} = \frac{d + \sqrt{2d-1}}{2(d^2 - 2d + 1)}
		\]
		We can bound the right-hand side by $d / (d-1)^2$ because $\sqrt{2d - 1} < d$.
	\end{proof}

	\subsection{Controlling the relative commutants}
	
	\begin{lem} \label{lem: first distance estimate}
		Let $\cA_n = 1_{M_n(\C)} \otimes \cM^{1/n}$ and let $\cA = \prod_{n \to \cU} \cA_n$.  Then almost surely, for all $a \in \cM^{\cU}$,
		\[
		d(a,\cA)^2 \leq \frac{3}{4} \sum_{j=1}^3 \norm{[X_j,a]}_2^2.
		\]
		In particular, $\cA = \{\mathbf{X}\}' \cap \cM^{\cU}$.
	\end{lem}
	
	\begin{proof}
		By Corollary \ref{cor: Hastings spectral gap} with $d = 3$, for sufficiently large $n$ and $A \in M_n(\C)$, we have almost surely
		\[
		\norm{A - \tr_n(A)}_2^2 \leq \frac{3}{4} \sum_{j=1}^3 \norm{[A,U_j^{(n)}]}_2^2.
		\]
		Because this is an inequality between linear operators on a Hilbert space, we may tensorize with the identity on $L^2(\cM^{1/n})$ (see e.g.\ \cite[Lemma 4.18]{GJKEP2023}), to obtain for $a \in M_n(\C) \otimes \cM^{1/n} = \cM$ that
		\[
		\norm{a - E_{\cA_n}[a]}_2^2 \leq \frac{3}{4} \sum_{j=1}^3 \norm{[a,X_j^{(n)}]}_2^2 \text{ for } a \in \cM.
		\]
		Then in the ultralimit, we obtain
		\[
		\norm{a - E_{\cA}[a]}_2^2 \leq \frac{3}{4} \sum_{j=1}^3 \norm{[a,X_j]}_2^2 \text{ for } a \in \cM^{\cU},
		\]
		since conditional expectations commute with ultraproducts.  This is the desired estimate for $\cA$.  For the final claim, $\cA \subseteq \{\mathbf{X}\}' \cap \cM^\cU$ is immediate from the construction of $\mathbf{X}$, and the opposite inclusion follows from the spectral gap estimate that we just proved.
	\end{proof}
	
	The analogous statement for $\mathbf{Y}$ is more delicate, and this is where we use the specific way that $\mathbf{X}$ and $\mathbf{Y}$ were constructed from $U_1^{(n)}$, \dots, $U_4^{(n)}$; this part of the argument was simplified due to the suggestion of Adrian Ioana and it is a close relative to the proof of \cite[Lemma~4.6]{ioana2021almost}.
	
	\begin{lem} \label{lem: second distance estimate} For a II$_1$ factor $\cM$ and for $\cB$ and $\mathbf{Y}$ as defined in \S\ref{sec: outline},  almost surely, for $b \in \cM^{\cU}$,
		\[
		d(b,\cB)^2 \leq 7 \sum_{j=1}^3 \norm{[Y_j,b]}_2^2.
		\]
		In particular, $\{\mathbf{Y}\}' \cap \cM^{\cU} = \cB$.
	\end{lem}
	
	\begin{proof}
		To prove the estimate for $\cB$, the same tensorization and ultralimit argument as in  the proof of Lemma~\ref{lem: first distance estimate} apply, and so it suffices to show that for $B = \begin{bmatrix} B_{1,1} & B_{1,2} \\ B_{2,1} & B_{2,2} \end{bmatrix} \in M_{2n}(\C)$, we have
		\begin{align*}
			\norm{\begin{bmatrix} B_{1,1} & B_{1,2} \\ B_{2,1} & B_{2,2} \end{bmatrix} - \begin{bmatrix} \tr_n(B_{1,1}) & 0 \\ 0 & \tr_n(B_{2,2}) \end{bmatrix}}_2^2
			&\leq 7 \sum_{j=1}^3 \norm{[B,Y_j^{(n)}]}_2^2 \\
			&= 7 \Bigg( \sum_{j=1}^2 \norm{ \begin{bmatrix}
					U_j^{(n)} B_{1,1} - B_{1,1} U_j^{(n)} & U_j^{(n)} B_{1,2} - B_{1,2} U_j^{(n)} \\
					U_j^{(n)} B_{2,1} - B_{2,1} U_j^{(n)} & U_j^{(n)} B_{2,2} - B_{2,2} U_j^{(n)}
			\end{bmatrix}}_2^2 \\
			&+ \norm{ \begin{bmatrix}
					U_3^{(n)} B_{1,1} - B_{1,1} U_3^{(n)} & U_3^{(n)} B_{1,2} - B_{1,2} U_4^{(n)} \\
					U_4^{(n)} B_{2,1} - B_{2,1} U_3^{(n)} & U_4^{(n)} B_{2,2} - B_{2,2} U_4^{(n)}
			\end{bmatrix}}_2^2 \Bigg)
		\end{align*}
		Equivalently, we want to show that
		\begin{align*}
			\norm{B_{1,1} - \tr_n(B_{1,1})}_2^2 &+ \norm{B_{2,2} - \tr_n(B_{2,2})}_2^2 + \norm{B_{1,2}}_2^2 + \norm{B_{2,1}}_2^2 \\
			&\leq 7 \Big( \sum_{j=1}^3 \norm{[B_{1,1},U_j^{(n)}]}_2^2 +  \sum_{j=1}^3  \norm{[B_{2,2},U_j^{(n)}]}_2^2 \\
			& \quad +  \sum_{j=1}^2 \norm{[B_{1,2},U_j^{(n)}]}_2^2 + \norm{U_3^{(n)} B_{1,2} - B_{1,2} U_4^{(n)}}_2^2 \\
			& \quad +  \sum_{j=1}^2 \norm{[B_{2,1},U_j^{(n)}]}_2^2 + \norm{U_4^{(n)} B_{2,1} - B_{2,1} U_3^{(n)}}_2^2 \Big)
		\end{align*}
		From Corollary \ref{cor: Hastings spectral gap}, we already know
		\[
		\norm{B_{1,1} - \tr_n(B_{1,1})}_2^2 \leq \frac{3}{4} \sum_{j=1}^3 \norm{[B_{1,1},U_j^{(n)}]}_2^2,
		\]
		and similarly for the $B_{2,2}$ term.  Thus, it remains to estimate the $B_{1,2}$ and $B_{2,1}$ terms.  We will handle the $B_{1,2}$ term and show that
		\begin{equation} \label{eq: desired B12 estimate}
			\norm{B_{1,2}}_2^2 \leq 7 \left( \sum_{j=1}^2 \norm{[B_{1,2},U_j^{(n)}]}_2^2 + \norm{U_3^{(n)} B_{1,2} - B_{1,2} U_4^{(n)}}_2^2 \right);
		\end{equation}
		the argument for the $B_{2,1}$ term is symmetrical.  First, we note that by Corollary \ref{cor: Hastings spectral gap} with $d = 2$, we have almost surely for sufficiently large $n$,
		\begin{equation}\label{eq.B12}
			\norm{B_{1,2} - \tr_n(B_{1,2})}_2^2 \leq 2 \sum_{j=1}^2 \norm{[B_{1,2},U_j^{(n)}]}_2^2.
		\end{equation}
		Thus, it remains to estimate $\tr_n(B_{1,2})$.  We note that
		\begin{align*}
			|\tr_n(B_{1,2})| \norm{U_3^{(n)} - U_4^{(n)}}_2 &= \norm{U_3^{(n)} \tr_n(B_{1,2}) - \tr_n(B_{1,2}) U_4^{(n)} }_2 \\
			&\leq \norm{U_3^{(n)}(B_{1,2} - \tr_n(B_{1,2})) - (B_{1,2} - \tr_n(B_{1,2})) U_4^{(n)}}_2 \\
			& \quad + \norm{U_3^{(n)} B_{1,2} - B_{1,2} U_4^{(n)}}_2 \\
			&\leq 2 \norm{B_{1,2} - \tr_n(B_{1,2})}_2 + \norm{U_3^{(n)} B_{1,2} - B_{1,2} U_4^{(n)}}_2.
		\end{align*}
		Note that $\mathbb{E} \tr_n((U_3^{(n)})^* U_4^{(n)}) = 0$, and so by Proposition \ref{prop: concentration}, we have $\tr_n((U_3^{(n)})^* U_4^{(n)}) \to 0$ almost surely, and thus $\norm{U_3^{(n)} - U_4^{(n)}}_2^2=\norm{1-(U_3^{(n)})^* U_4^{(n)}}_2^2 \to 2$ almost surely, and hence is eventually larger than $9/5$.
		Hence, we have that for sufficiently large $n$,
		\[
		|\tr_n(B_{1,2})| \leq  \sqrt{5/9} \left( 2 \norm{B_{1,2} - \tr_n(B_{1,2})}_2 + \norm{U_3^{(n)} B_{1,2} - B_{1,2} U_4^{(n)}}_2 \right).
		\]
		By the Cauchy-Schwarz inequality and our previous estimate for $\norm{B_{1,2} - \tr_n(B_{1,2})}_2$,
		\begin{align*}
			|\tr_n(B_{1,2})|^2 &\leq  \frac{5}{9} (1 + 1/8) \left( 4 \norm{B_{1,2} - \tr_n(B_{1,2})}_2^2 + 8 \norm{U_3^{(n)} B_{1,2} - B_{1,2} U_4^{(n)}}_2^2 \right) \\
			&\leq  \frac{5}{8} \left( 4 \cdot 2 \sum_{j=1}^2 \norm{[B_{1,2},U_j^{(n)}]}_2^2 + 8 \norm{U_3^{(n)} B_{1,2} - B_{1,2} U_4^{(n)}}_2^2 \right) \\
			&= 5 \sum_{j=1}^2 \norm{[B_{1,2},U_j^{(n)}]}_2^2 + 5 \norm{U_3^{(n)} B_{1,2} - B_{1,2} U_4^{(n)}}_2^2
		\end{align*}
		Hence, using this and \eqref{eq.B12},
		\begin{align*}
			\norm{B_{1,2}}_2^2 &= \norm{B_{1,2} - \tr_n(B_{1,2})}_2^2 + \tr_n(B_{1,2})^2 \\
			&\leq  \left(2 + 5 \right) \sum_{j=1}^2 \norm{[B_{1,2},U_j^{(n)}]}_2^2 + 5 \norm{U_3^{(n)} B_{1,2} - B_{1,2} U_4^{(n)}}_2^2 \\
			&\leq 7 \left( \sum_{j=1}^2 \norm{[B_{1,2},U_j^{(n)}]}_2^2 + \norm{U_3^{(n)} B_{1,2} - B_{1,2} U_4^{(n)}}_2^2 \right)
		\end{align*}
		as desired.
	\end{proof}
	
	\begin{remark}
		In the spirit of Goldbring's work on spectral gap and definability \cite{Goldbring2023spectralgap}, our bound on the distance to the relative commutant in Lemma \ref{lem: first distance estimate} shows that $\cA = \{\mathbf{X}\}' \cap \cM^\cU$ is a definable set with parameters $\mathbf{X}$ (see \S \ref{sec: definable}).  Similarly, Lemma \ref{lem: second distance estimate} implies that $\{\mathbf{Y}\}' \cap \cM^{\cU}$ is definable with parameters $\mathbf{Y}$.
	\end{remark}
	
	\subsection{Conclusion of the proof of Theorem \ref{thm: main MC} in the $\mathrm{II}_1$ factor case} \label{sec: MC factor proof conclusion}
	
	\begin{proof}[Proof of Theorem \ref{thm: main MC} in the $\mathrm{II}_1$ factor case]
		Referring to the outline of the proof stated in \S\ref{sec: outline}, we have shown (1) in Lemma \ref{lem: inf formula limit}, (2) in Lemma \ref{lem: first distance estimate}, and (3) in Lemma \ref{lem: second distance estimate}.  Item (1) shows that, almost surely, $\phi^{\cM^{\cU}}(\mathbf{X}) \leq \phi^{\cM^{\cU}}(\mathbf{Y})$ for all $\inf$-formulas.  If $\cM$ were model complete, then $\mathbf{X}$ and $\mathbf{Y}$ would have the same type by Lemma~\ref{lem: type MC}.  Hence, to finish the argument, it suffices to show that $\mathbf{X}$ and $\mathbf{Y}$ do not have the same type.
		
		In fact, we claim that $\mathbf{X}$ and $\mathbf{Y}$ do not even have the same two-quantifier type.  Consider the formula
		\begin{multline*}
			\psi(x_1, x_2, x_3) = \\ \inf_{z_1} \left[ 1 - \norm{z_1}_2^2 + |\tr(z_1)|^2 + 7 \sum_{j=1}^3 \norm{[x_j,z_1]}_2^2 + \sup_{z_2} \left[ \norm{[z_1,z_2]}_2^2 \dot{-} 28 \sum_{j=1}^3 \norm{[x_j,z_2]}_2^2 \right] \right],
		\end{multline*}
		where $z_1$ and $z_2$ range over the unit ball. Then the condition $\psi(x_1,x_2, x_3)=0$ attempts to assert the existence of $z_1$ with $\|z_1\|_2=1$ and $\tr(z_1)=0$ such that $z_1$ commutes with $x_j$ for $j = 1,2,3$ and also commutes with every $z_2$ in the relative commutant of $\{x_1,x_2, x_3\}$.\footnote{The statement does not literally assert this, but it asserts the first two statements in an approximate sense, and the last part is necessarily imperfect because there is no implication in continuous logic, but we will see that it serves the purpose.}  We will find a self-adjoint unitary $z_1$ that commutes with $Y_j$ for $j = 1,2,3$, has zero trace, and commutes with everything in the relative commutant of $\{Y_1,Y_2, Y_3\}$; this will suffice to show that $\psi^{\cM^\cU}(\mathbf{Y})=0$. 
		
		Indeed, $\{\bf{Y}\}' \cap \cM^\cU = \cB$ is a direct sum of two copies of $\prod_{n \to \cU} \cM^{1/2n}$ so it has a central projection $p$ of trace $1/2$.  Let $z_1 = 2p - 1$, so that $\norm{z_1}_2^2 = 1$ and $\tr(z_1) = 0$ and $\sum_{j=1}^3 \norm{[Y_j,z_1]}_2 = 0$.  Also for every $z_2$, we have
		\[
		\norm{[z_1,z_2]}_2^2 \leq \left( \norm{[z_1,E_{\cB}[z_2]]}_2 + 2 \norm{z_1} d(z_2,\cB) \right)^2 = 4 d(z_2,\cB)^2 \leq 28 \sum_{j=1}^3 \norm{[Y_j,z_2]}_2^2,
		\]
		because of Lemma \ref{lem: second distance estimate} and the fact that $[z_1,E_{\cB}[z_2]] = 0$.
		
		On the other hand, we claim that $\psi^{\cM^{\cU}}(\mathbf{X}) = 1$.  Because the ultraproduct $\cM^{\cU}$ is countably saturated (see \S \ref{sec: model theory prelims}), there is some $z_1 \in \cM^{\cU}$ that attains the infimum in the formula.  Let $z_1' = E_{\cA}[z_1]$.  Because $\cA$ is a factor, a Dixmier averaging argument (see e.g.,  \cite[Lemma 4.2]{FHS2013}) shows that
		\[
		\norm{z_1'}_2^2 - |\tr(z_1')|^2 
		=\|z-\tr(z)\|_2^2\leq \sup_{z_2 \in \cA_1} \norm{[z_1',z_2]}_2^2,
		\]
		where $\cA_1$ is the unit ball of $\cA$.  Using choices of $z_2 \in \cA$ witnessing this inequality as candidates for the supremum in $\psi$, we conclude
		\[
		\psi^{\cM^{\cU}}(\mathbf{X}) \geq 1 - \norm{z_1}_2^2 + |\tr(z_1)|^2 + d(z_1,\cA)^2 + \left[ \norm{z_1'}_2^2 - |\tr(z_1')|^2 + 0 \right],
		\]
		where we have also applied the spectral gap inequality from Lemma \ref{lem: first distance estimate} to get the $d(z_1, \cA)^2$ term.  Noting that $d(z_1,\cA)^2 = \norm{z_1 - z_1'}_2^2 = \norm{z_1}_2^2 - \norm{z_1'}_2^2$ and that $\tr(z_1') = \tr(z_1)$, the entire expression evaluates to $1$.  For the upper bound $\psi^{\cM^{\cU}}(\mathbf{X}) \leq 1$, simply take $z_1 = 0$.

	\end{proof}
	
	\begin{remark}[Lack of quantifier elimination for $\mathrm{II}_1$ factors] \label{rem: II1 QE}
		Our argument also gives another proof of \cite[Theorem~1]{Farah2023}, that a $\mathrm{II}_1$ factor never admits quantifier elimination, even without the assumption that $M_2(\cM)$ embeds into $\cM^{\cU}$.  Indeed, this assumption was only used to relate the existential types of $\mathbf{X}$ and $\mathbf{Y}$.  It is immediate from Lemma \ref{lem: type convergence} that the quantifier-free type of $(U_1^{(n)},U_2^{(n)},U_3^{(n)})$ converges almost surely as $n \to \cU$, and the quantifier-free type of $(U_1^{(n)},U_2^{(n)},U_4^{(n)})$ converges to the same limit, hence so does the quantifier-free type of $(U_1^{(n)}\oplus U_1^{(n)},U_2^{(n)} \oplus U_2^{(n)},U_3^{(n)} \oplus U_4^{(n)})$.  Therefore, $\mathbf{X}$ and $\mathbf{Y}$ have the same quantifier-free type almost surely.  In fact, by Voiculescu's asymptotic freeness theory \cite{Voiculescu1991,Voiculescu1998}, $\mathbf{X}$ and $\mathbf{Y}$ are almost surely triples of freely independent unitaries whose spectral measures are uniform over the circle.  However, the argument given above shows that $\mathbf{X}$ and $\mathbf{Y}$ do not have the same type, so that $\cM$ does not admit quantifier elimination.  
	\end{remark}
	
	\begin{remark}[Alternative approaches to the proof]
		Theorem \ref{thm: main MC} in the $\mathrm{II}_1$ factor case can be proved in various ways using other constructions of quantum expanders, similar to how IF used spectral gap property (T) groups to show a lack of quantifier elimination for $\mathrm{II}_1$ factors in \cite[Lemma~2.1]{Farah2023}.  Let $U_1^{(n)}$, \dots, $U_d^{(n)}$ be a sequence of deterministic matrices such that $U_1^{(n)}$, \dots, $U_d^{(n)}$ and their adjoints are a $(2d,\epsilon)$-quantum expander.  Let $U_{d+1}^{(n)}$ and $U_{d+2}^{(n)}$ be independent Haar random unitaries.  Then the above argument for Theorem \ref{thm: main MC} in the factor case could also be done using
		\[
		\mathbf{X}^{(n)} = (U_1^{(n)} \otimes 1_{\cM^{1/n}}, \dots, U_d^{(n)} \otimes 1_{\cM^{1/n}}, U_{d+1}^{(n)} \otimes 1_{\cM^{1/n}}),
		\]
		and
		\[
		\mathbf{Y}^{(n)} = ((U_1^{(n)} \oplus U_1^{(n)}) \otimes 1_{\cM^{1/2n}}, \dots, (U_d^{(n)} \oplus U_d^{(n)}) \otimes 1_{\cM^{1/2n}}, (U_{d+1}^{(n)} \oplus U_{d+2}^{(n)}) \otimes 1_{\cM^{1/2n}}).
		\]
		Indeed, concentration of measure (Proposition \ref{prop: concentration} and the proof of  Lemma \ref{lem: type convergence}) still apply to a mixture of deterministic matrices and Haar random unitaries, and hence Lemma \ref{lem: inf formula limit} still goes through.  The arguments for Lemma \ref{lem: first distance estimate} and \ref{lem: second distance estimate} only use the fact that $U_1^{(n)}$, \dots, $U_d^{(n)}$ is an expander and that $\norm{U_{d+1}^{(n)} - U_{d+2}^{(n)}}_2$ converges to $\sqrt{2}$ as $n \to \infty$.  Further comments on alternative proofs can be found in the first arXiv version of this paper.
	\end{remark}
	
	\section{Model completeness for tracial von Neumann algebras} \label{sec: MC general case}
	
	It is now straightforward to extend Theorem \ref{thm: main MC} from II$_1$ factors to arbitrary tracial von Neumann algebras as outlined in the introduction.
	
	\subsection{Model completeness and direct sums} \label{sec: MC direct sums}
	
	\begin{lem} \label{lem: direct sum} If the theory of a tracial von Neumann algebra $\cM$ is model-complete, then the theory of every direct summand of $\cM$ is model-complete. 
	\end{lem}
	
	\begin{proof} 	
		Let $\cM$ be a tracial von Neumann algebra which decomposes as a direct sum $\cM_1 \oplus \cM_2$ with weights $\alpha$ and $1 - \alpha$.  Assume the theory of $\cM$ is model complete; we will prove that  the theory of each one of $\cM_1$ and $\cM_2$ is model complete.
		
		Let $\cN_1 \equiv \cM_1$ and $\cN_2 \equiv \cM_2$, and $\iota_1: \cM_1 \to \cN_1$ and $\iota_2: \cM_2 \to \cN_2$ be trace-preserving $*$-homomorphisms; we need to show that $\iota_1$ and $\iota_2$ are elementary.  Let $\cN$ be the direct sum of $\cN_1$ and $\cN_2$ with weights $\alpha$ and $1 - \alpha$. Note that by \cite{FG2023}, $\cN \equiv \cM$ since the theory of $\cN$ is uniquely determined by the theories of the direct summands.  By model completeness of $\cM$, the map $\iota = \iota_1 \oplus \iota_2: \cM \to \cN$ is elementary.
		
		Let $\phi(x_1,\dots,x_n)$ be an $\Lang_{\tr}$-formula, and we will show that $\phi^{\cN_1}(\iota_1(\mathbf{a})) = \phi^{\cM_1}(\mathbf{a})$ for ${\mathbf{a} = (a_1,\dots,a_n) \in \cM_1^n}$.  Because prenex formulas are dense in the space of all formulas \cite[\S 6]{BYBHU2008}, assume without loss of generality that
		\[
		\phi(x_1,\dots,x_n) = \inf_{y_1} \sup_{y_2} \dots \inf_{y_{2m-1}} \sup_{y_{2m}} F(\re \tr(p_1(\mathbf{x},\mathbf{y})),\dots,\re \tr(p_k(\mathbf{x},\mathbf{y})))
		\]
		where $y_1$, \dots, $y_{2m}$ are variables in the unit ball, $F: \R^k \to \R$ is continuous, and $p_1$, \dots, $p_k$ are non-commutative $*$-polynomials.  Define
		\[
		\psi(x_1,\dots,\tilde{x}_n,z) = \inf_{y_1} \sup_{y_2} \dots \inf_{y_{2m-1}} \sup_{y_{2m}} F\left(\frac{1}{\alpha} \re \tr(p_1(\mathbf{x},z\mathbf{y})),\dots, \frac{1}{\alpha} \re \tr(p_k(\mathbf{x},z\mathbf{y})) \right),
		\]
		where $z\mathbf{y} = (zy_1,\dots,zy_{2m})$.  Observe that
		\[
		\phi^{\cM_1}(a_1,\dots,a_n) = \psi^{\cM}(a_1 \oplus 0,\dots,a_n \oplus 0, 1 \oplus 0),
		\]
		because $(1 \oplus 0)(y \oplus y') = y \oplus 0$.  Similarly,
		\[
		\phi^{\cN_1}(\iota_1(a_1),\dots,\iota_1(a_n)) = \psi^{\cN}(\iota(a_1 \oplus 0),\dots,\iota(a_n \oplus 0), \iota(1 \oplus 0)).
		\]
		The mapping $\iota: \cM \to \cN$ is elementary, and hence
		\[
		\psi^{\cN}(\iota(a_1 \oplus 0),\dots,\iota(a_n \oplus 0), \iota(1 \oplus 0)) = \psi^{\cM}(a_1 \oplus 0,\dots,a_n \oplus 0, 1 \oplus 0).
		\]
		This shows $\phi^{\cN_1}(\iota_1(\mathbf{a})) = \phi^{\cM_1}(\mathbf{a})$, so the mapping $\iota_1$ is elementary as desired. The same argument applies to $\iota_2$.  Therefore, $\cM_1$ and $\cM_2$ are model complete.
	\end{proof}
	
	\begin{remark} \label{rem: QE direct sums}
		Similarly, if  $\cM = (\cM_1,\alpha) \oplus (\cM_2,1-\alpha)$ and if $\Th(\cM)$ admits quantifier elimination, then $\Th(\cM_j)$ admits quantifier elimination for $j = 1$, $2$.  To see this, consider $n$-tuples $\mathbf{x}$ and $\mathbf{y}$ in $\cM_1$ that have the same quantifier-free type in $\cM_1$ (i.e.\ they have the same $*$-moments).  Then $(x_1 \oplus 0, \dots, x_n \oplus 0, 1 \oplus 0)$ and $(y_1 \oplus 0, \dots, y_n \oplus 0,1\oplus 0)$ have the same quantifier-free type in $\cM$.  Therefore, by Lemma \ref{lem: type QE}, they have the same type in $\cM$.  As we saw above, for each formula $\phi$, there exists $\psi$ such that $\phi^{\cM_1}(x_1,\dots,x_n) = \psi^{\cM}(x_1 \oplus 0, \dots, x_n \oplus 0, 1 \oplus 0)$ (and similarly for the $y_j$'s), and hence $\mathbf{x}$ and $\mathbf{y}$ have the same type in $\cM_1$, and so $\Th(\cM_1)$ has quantifier elimination by Lemma \ref{lem: type QE}.
	\end{remark}
	
	\begin{remark}
		The relationship between model theoretic properties and direct sums/integrals is an important topic of recent study; \cite{FG2023} showed how to determine the theory of the direct integral from that of the integrands, and the opposite direction was studied for von Neumann algebras in \cite{GaoJekelIntegral}, both of which are now special cases of the general theory of direct integrals developed by Ben Yaacov, Ibarluc{\'\i}a, and Tsankov \cite{BYIT2024}.  Based on these works, it is plausible that model completeness of a direct \emph{integral} implies model completeness of the integrands almost everywhere in general, but we leave this as a question for future research.
	\end{remark}
	
	\subsection{Conclusion of the proof of {Theorem \ref{thm: main MC}}}
	
	By Lemma~\ref{lem: direct sum}, because we already proved Theorem \ref{thm: main MC} in the case of $\mathrm{II}_1$ factors, we can eliminate any direct summands that are $\mathrm{II}_1$ factors satisfying that $M_2(\cM_\omega)$ embeds into $\cM_\omega^{\cU}$.  It remains to handle the diffuse part of the direct integral decomposition for $\cM$, which actually turns out to be much easier.
	
	\begin{lem} \label{lem: diffuse MC}
		Let $(\cM,\tau) = \int_{[0,1]} (\cM_\omega,\tau_\omega)\,d\omega$, where $\cM_\omega$ is a separable $\mathrm{II}_1$ factor such that $M_2(\cM_\omega)$ embeds into $\cM_\omega^{\cU}$.  Then $\cM$ is not model complete.
	\end{lem}
	
	\begin{proof}
		Let $\cN = L^\infty[0,1] \otimes \cM$.  Note that
		\[
		\cN = \int_{[0,1]^2} \cM_\omega \,d\omega \,d\omega'.
		\]
		Thus, the distribution of $\Th(\cM_\omega)$ over $[0,1]^2$ is the same as the distribution of the $\Th(\cM_\omega)$ over $[0,1]$.  Therefore, it follows from \cite[Theorem 2.3]{FG2023} that $\cM \equiv \cN$.  Moreover, $\cN \oplus \cN \cong \cN$.  Now fix an ultrafilter $\cU$ on $\N$ and note that $M_2(\cM_\omega)$ embeds into $\cM_\omega^{\cU}$ for all $\omega$, hence $M_2(\cN)$ embeds into $\cN^{\cU}$.  Consider a trace preserving $*$-homomorphism
		\[
		\cN \to \cN \oplus \cN \to M_2(\cN) \to \cN^{\cU},
		\]
		where the first map is an isomorphism and the second map is the block diagonal embedding.  Then $1 \oplus 0$ is central in $\cN \oplus \cN$ but $1 \oplus 0$ is not central in $M_2(\cN)$.  Hence, our homomorphism does not map $Z(\cN)$ into $Z(\cN^{\cU})$, so it is not elementary.
	\end{proof}
	
	\begin{proof}[Proof of Theorem \ref{thm: main MC}]
		Suppose $\cM$ has a direct integral decomposition where $\cM_\omega$ is a $\mathrm{II}_1$ factor such that $M_2(\cM_\omega)$ embeds $\cM_{\omega}^{\cU}$, for $\omega$ in some positive measure set.  If the positive measure set has an atom, then $\cM$ has a direct summand $\cN$ which is a $\mathrm{II}_1$ factor such that $M_2(\cN)$ embeds into $\cN^{\cU}$.  The results of the previous section show that $\cN$ is not model complete, hence by Lemma \ref{lem: direct sum}, $\cM$ is not model complete.
		
		If there is no atom in our positive measure set, then $\cM$ has a direct summand of the form $\cN = \int_{[0,1]} \cN_\alpha\,d\alpha$ where the integral occurs with respect to Lebesgue measure and $\cN_\alpha$ is a $\mathrm{II}_1$ factor such that $M_2(\cN_\alpha)$ embeds into $\cN_\alpha^{\cU}$.  Hence, by Lemma \ref{lem: diffuse MC}, $\cN$ is not model complete, and so by Lemma \ref{lem: direct sum}, $\cM$ is not model complete.
	\end{proof}
	
	\begin{remark} \label{rem: alternate QE proof}
		A similar argument recovers the result of the first author that the theory of any separable tracial von Neumann algebra with a type $\mathrm{II}_1$ summand never admits quantifier elimination \cite{Farah2023}.  An algebra satisfying the assumptions of Theorem \ref{thm: main MC} either has a $\mathrm{II}_1$ factor as a direct summand, or it has a type $\mathrm{II}_1$ direct summand with diffuse center.  If there is a type $\mathrm{II}_1$ direct summand $\cN$, then $\Th(\cN)$ does not have quantifier elimination by Remark \ref{rem: II1 QE} and hence by Remark \ref{rem: QE direct sums}, $\Th(\cM)$ does not have quantifier elimination.  On the other hand, suppose $\cN$ is a type $\mathrm{II}_1$ direct summand of $\cM$ with diffuse center.  In this case, we argue similarly to Lemma \ref{lem: eliminate diffuse matrix term}; $\cN$ has a central projection of trace $1/2$, and also a non-central projection of trace $1/2$, 
		and hence $\Th(\cN)$ does not have quantifier elimination.  So by Remark \ref{rem: QE direct sums}, $\Th(\cM)$ does not have quantifier elimination.
	\end{remark}
	
	\section{Further remarks} \label{sec: further remarks}
	
	\subsection{Topological properties} \label{sec: topological}
	
	In this section, we study the topological properties of the set of theories that admit quantifier elimination (and those that are model complete), and in particular we will see that quantifier elimination is generic among purely atomic tracial von Neumann algebras (though a lack of quantifier elimination is generic for tracial von Neumann algebras in general).
	
	There is a natural topology on the space of complete theories, where basic open sets have the form
	\[
	\{ \mathrm{T} \models |\phi_1 - c_1| < \epsilon_1, \dots, |\phi_k - c_k| < \epsilon_k \}
	\]
	for some finite list of formulas $\phi_1$, \dots, $\phi_k$, real numbers $c_1$, \dots, $c_k$, and positive $\epsilon_1$, \dots, $\epsilon_k$.  In fact, this topology can be understood in functional analytic terms as follows.  The sentences of a fixed language $\Lang$ form a real algebra that has a natural norm (see the last sentence of \cite[Definition D.2.4]{Fa:STCstar}). A complete theory in language $\Lang$ is naturally identified with a bounded homomorphism from this algebra into $\R$ (\cite[Definition~D.2.8]{Fa:STCstar}), and the topology on the space of complete theories then agrees with the weak-$*$ topology.  The space of theories is metrizable whenever the language $\Lang$ is separable (which is the case for tracial von Neumann algebras).  Moreover, if $\mathcal{C}$ is a class of $\Lang$-structures that is closed under elementary equivalence, then $\mathcal{C}$ is axiomatizable if and only if $\Th_{\mathcal{C}}=\{\Th(\cM): \cM \in \mathcal{C}\}$ is a closed set and every model of some theory in  $\Th_{\mathcal{C}}$ belongs to $\mathcal{C}$. 
	
	A very basic observation is that quantifier elimination and model completeness define sets that are neither open nor closed in the space of theories of tracial von Neumann algebras.
	
	\begin{prop}\label{P.dense}
		The following sets of theories of tracial von Neumann algebras are not closed (equivalently, the corresponding classes are not axiomatizable):
		\begin{enumerate}[(1)]
			\item Those which admit quantifier elimination.
			\item Those which do not admit quantifier elimination.
			\item Those which are model complete.
			\item Those which are not model complete.
		\end{enumerate}
	\end{prop}
	
	\begin{proof}
		We use the following observation several times:  For any two tracial von Neumann algebras $\cM_0$ and $\cM_1$, the theory of $\cM_\alpha = (\cM_0,1-\alpha) \oplus (\cM_1,\alpha)$ depends continuously on $\alpha \in [0,1]$.  This idea was used in \cite[Proposition 5.1]{GH2016}.  Indeed, one can show by induction that for each formula $\phi$, the quantity $\phi^{\cM_\alpha}(x_1 \oplus x_1', \dots, x_n \oplus x_n')$ is continuous in $\alpha$ uniformly over $x_j$ and $x_j'$ in the unit ball.
		
		Now we proceed to the main claims:
		\begin{enumerate}[(1)]
			\item $M_n(\C)$ admits quantifier elimination.  Fixing an ultrafilter $\cU$ on the natural numbers, $\lim_{n \to \cU} \Th(M_n(\C)) = \Th(\prod_{n \to \cU} M_n(\C))$, which does not admit quantifier elimination by \cite{Farah2023} since the matrix ultraproduct is a $\mathrm{II}_1$ factor.\footnote{This also follows from \cite[\S 3]{GHS2013} since the matrix ultraproduct is Connes embeddable and not elementarily equivalent to $\cR$, because it does not have property Gamma.}
			\item Consider $(M_n(\C),1-\alpha) \oplus (\cR,\alpha)$.  This does not admit quantifier elimination when $\alpha > 0$ but does admit quantifier elimination when $\alpha = 0$.
			\item This follows from the same argument as (1).
			\item This follows from the same argument as (2) since $(M_n(\C),1-\alpha) \oplus (\cR,\alpha)$ is not model complete by Theorem \ref{thm: main MC}.  \qedhere
		\end{enumerate}
	\end{proof}
	
	While the sets of theories defined by quantifier elimination and model completeness are not open or closed, they are $G_\delta$-sets. In fact, this holds for separable metric languages in general.  We remark that the analogous statement also holds for countable languages in discrete model theory (and the analog of Proposition~\ref{P.dense} is true for some languages).  Hence, the descriptive complexity of these sets does not increase when we pass from discrete structures to metric structures (in stark contrast, there is a bizarre increase in complexity for sets of omissible types \cite{farah2018omitting}).
	
	\begin{prop}\label{P.Gdelta}
		Let $\Lang$ be a separable language of metric structures. Both  the set of complete theories that admit quantifier elimination and the set of complete theories that are model complete are $G_\delta$ sets.
	\end{prop}
	
	\begin{proof}
		Consider quantifier elimination.  Since the language is separable, choose for each $n$ a countable dense set $\mathcal{F}_n$ of formulas in $n$ variables (if there are multiple sorts, then we choose such a set for each tuple of sorts).  For each $n$ and $\phi \in \mathcal{F}_n$, for each $k \geq 1$, let $G_{\phi,k}$ be the set of complete theories $\mathrm{T}$ such that there exists a quantifier-free formula $\psi$ such that $\mathrm{T}$ models
		\[
		\sup_{x_1,\dots,x_n} |\phi(x_1,\dots,x_n) - \psi(x_1,\dots,x_n)| < \frac{1}{k}.
		\]
		Then $G_{\phi,k}$ is open and $\bigcap_{\phi, k} G_{\phi,k}$ is precisely the set of theories that admit quantifier elimination, since it suffices to approximate a \emph{dense} subset of formulas by quantifier-free formulas.  The argument for model completeness works the same way using Lemma \ref{lem: type MC} (3).
	\end{proof}
	
	So the set of theories of tracial von Neumann algebras with quantifier elimination is non-closed, non-open, and $G_\delta$.  We now show it is meager, since in fact the set of theories of type I von Neumann algebras is meager.  Our proof goes by way of spectral gap.
	
	\begin{lem} \label{lem: spectral gap closed}
		Let $d \in \N$ and $C > 0$.  The complete theories of tracial von Neumann algebras with $(C,d)$-spectral gap form a closed set with dense complement.
	\end{lem}
	
	\begin{proof}
		By \cite[Lemma 4.2]{FHS2013}, the center $Z(\cM)$ is definable relative to the theory of tracial von Neumann algebras.  Hence, similar to \cite[Definition 3.2.3; Lemma 3.2.5]{FHLRTVW2021} in the $\mathrm{C}^*$-algebra case, $d(y,Z(\cM))^2$ is a definable predicate (or it is a formula in an expanded language with a sort added for $Z(\cM)$).  Thus, consider the sentence
		\[
		\inf_{x_1,\dots,x_d \in B_1^{\cM}} \sup_{y \in B_1^{\cM}} \left( d(y,Z(\cM))^2 \dot{-} C \sum_{j=1}^d \norm{[x_j,y]}_2^2 \right) = 0.
		\]
		Note that $\cM$ has $(C,d)$-spectral gap, then $\cM$ satisfies this sentence. The converse holds when $\cM$ is countably saturated because we can choose some $x_1$, \dots, $x_d$ that realize the infimum.  Since every complete theory had a countably saturated model, the set of theories of von Neumann algebras with $(C,d)$-spectral gap is equal to the set of theories satisfying this sentence, hence is closed.  To see that its complement is dense, note that for every tracial von Neumann algebra $\cM$, the direct sum $(\cM,1-\alpha) \oplus (\cR,\alpha)$ does not have spectral gap, and $\Th((\cM,1-\alpha) \oplus (\cR,\alpha)) \to \cM$ as $\alpha \to 0$.
	\end{proof}
	
	\begin{prop}
		The following properties define meager sets in the space of complete theories of tracial von Neumann algebras.
		\begin{enumerate}[(1)]
			\item Tracial von Neumann algebras with spectral gap.
			\item Type $\mathrm{I}$ tracial von Neumann algebras.
			\item Tracial von Neumann algebras whose theory admits quantifier elimination.
		\end{enumerate}
	\end{prop}
	
	\begin{proof}
		(1) By Lemma \ref{lem: spectral gap closed} the $(C,d)$-spectral gap property defines a closed set whose complement is dense.  Taking the union over $C$ and $d$ in $\N$ yields a meager $F_\sigma$ set.
		
		(2) Hastings's result (see Theorem \ref{thm: Hastings} and Corollary \ref{cor: Hastings spectral gap} above) shows that matrix algebras $M_n(\C)$ have spectral gap for a fixed $C$ and $d$ (for instance one can take $d = 2$ and $C = 2$).  It is straightforward to check that a direct integral of tracial von Neumann algebras with $(2,2)$-spectral gap also has $(2,2)$-spectral gap.  Hence, all separable type $\mathrm{I}$ tracial von Neumann algebras have $(2,2)$-spectral gap, so their theories are contained in the meager set from (1).
		
		(3) Quantifier elimination can only hold for the theories of type I tracial von Neumann algebras \cite[Theorem 1]{Farah2023}.
	\end{proof}
	
	As the set of theories of von Neumann algebras with quantifier elimination is meager in the space of all theories, we now consider its topological properties \emph{within} the space of theories of type I von Neumann algebras.  In light of Theorem \ref{thm: main QE}, tracial von Neumann algebras $\cM$ whose theories admit quantifier elimination come in two varieties, those with an $L^\infty[0,1]$ summand and those without.  First, those $\cM$ with an $L^\infty[0,1]$ summand can only have finitely many matrix algebra summands, since projections in the atomic part cannot have trace smaller than the weight $\alpha_0$ of the $L^\infty[0,1]$ summand by Proposition~\ref{prop: QE obstructions} (3).  Fix natural numbers $k$ and $n_1$, \dots, $n_k$, and consider
	\[
	\cM = (L^\infty[0,1],\alpha_0) \oplus \bigoplus_{j=1}^k (M_{n_j}(\C),\alpha_j).
	\]
	From \ref{prop: QE linear inequality}, we can see that the set of weights $(\alpha_0,\dots,\alpha_k)$ such that $\cM$ admits quantifier elimination is an open subset of the $k$-simplex, as we can see from Proposition \ref{prop: QE linear inequality}.  However, it is not dense since everything in the closure must satisfy $\alpha_j/n_j \geq \alpha_0$ for $j \geq 1$.
	
	Second, we have purely atomic $\cM$.  As noted in \cite[\S 3]{FG2023}, purely atomic algebras can be parameterized by $\rho_{\cM}(m,n)$ for $m, n \geq 1$, where for each $m \in \N$, the values $\rho_{\cM}(m,1) \geq \rho_{\cM}(m,2) \geq \dots$ are the weights of the central projections associated to $M_m(\C)$ terms in the direct sum decomposition.  If there are only finitely many $M_m(\C)$ terms, we set $\rho_{\cM}(m,n) = 0$ for $n$ larger than the number of such terms.  Let
	\[
	\Delta = \left\{ (\alpha_{m,n})_{m,n\geq 1}: \alpha_{m,n} \geq \alpha_{m,n+1} \geq 0, \sum_{m,n \geq 1} \alpha_{m,n} = 1 \right\}.
	\]
	We view $\Delta$ as a metric space with respect to the $L^1$ metric.  The resulting topology on $\Delta$ agrees with the topology of pointwise convergence (however, $\Delta$ is not compact because elements of $\Delta$ can converge pointwise to zero).
	
	\begin{lem}
		For $\vec{\alpha} = (\alpha_{m,n})_{m,n \geq 1}$, let
		\[
		\cM_{\vec{\alpha}} = \bigoplus_{m,n \geq 1} (M_m(\C), \alpha_{m,n})
		\]
		be the associated purely atomic tracial von Neumann algebra.  The map $\vec{\alpha} \mapsto \Th(\cM_{\vec{\alpha}})$ is a homeomorphism onto its image.
	\end{lem}
	
	\begin{proof}
		\cite[Theorem 2.3]{FG2023} implies that the theory of $\cM_{\vec{\alpha}}$ depends continuously on the weights $\vec{\alpha}$.  The construction in \cite[Lemma 3.2]{FG2023} shows that $\alpha_{m,n} = \rho_{\cM_{\vec{\alpha}}}(m,n)$ can be recovered from $\Th(\cM_{\vec{\alpha}})$.  In particular, one can see from this that for each $m, n \geq 1$, if $\vec{\alpha} \in \Delta$ and the theory of $\cN$ is sufficiently close to that of $\cM_{\vec{\alpha}}$, then $\rho_{\cN}(m,n)$ will be close to $\alpha_{m,n}$.
	\end{proof}
	
	\begin{prop}
		The set of $\vec{\alpha} \in \Delta$ such that $\Th(\cM_{\vec{\alpha}})$ has quantifier elimination is comeager.
	\end{prop}
	
	\begin{proof}
		Let $\cM_{\vec{\alpha},k} = \bigoplus_{1 \leq m,n \leq k} (M_m(\C),\alpha_{m,n}) \subseteq \cM_{\vec{\alpha}}$, and let $\cM_{\vec{\alpha},k}^{\perp}$ be the direct sum over the complementary indices.   Let $\tau_{\vec{\alpha}}$ be the trace on $\cM_{\vec{\alpha}}$.  Let
		\[
		\epsilon_k(\vec{\alpha}) = \min \{ |\tau_{\vec{\alpha}}(p) - \tau_{\vec{\alpha}}(q)|: p, q \text{ projections in } \cM_{\vec{\alpha},k} \text{ with } \tau(p) \neq \tau(q) \}. 
		\]
		Let
		\[
		G_k = \{ \vec{\alpha} \in \Delta:  1 - \sum_{1 \leq m,n \leq k} \alpha_{m,n} < \epsilon_k(\vec{\alpha}) \}.
		\]
		Note that $G_k$ is open in $\Delta$, hence also $\bigcup_{k \geq \ell} G_k$ is open.  Moreover, contains the set of $\vec{\alpha}$ such that $\vec{\alpha}$ is supported on $\{1,\dots,k\}^2$, and so $\bigcup_{k \geq \ell} G_k$ is dense. 
		Therefore,
		\[
		G = \bigcap_{\ell \in \N} \bigcup_{k \geq \ell} G_k
		\]
		is comeager.  Furthermore,
		\[
		F = \{\vec{\alpha} \in \Delta: \alpha_{m,n} \text{ are linearly independent over } \Q\}
		\]
		is comeager because non-vanishing of $\Q$-linear combinations is a countable family of open conditions. Hence, $F \cap G$ is comeager.
		
		We claim that if $\vec{\alpha} \in F \cap G$, then $\cM_{\vec{\alpha}}$ admits quantifier elimination.  Let $p$ and $q$ be projections of the same trace in $\cM_{\vec{\alpha}}$.  For each $k$, write $p = p_k \oplus p_k^{\perp}$ and $q = q_k \oplus q_k^{\perp}$ with respect to the decomposition $\cM_{\vec{\alpha}} = \cM_{\vec{\alpha},k} \oplus \cM_{\vec{\alpha},k}^{\perp}$.  If $\vec{\alpha} \in G_k$, then by construction of $G_k$, we have
		\[
		|\tau_{\vec{\alpha}}(p_k) - \tau_{\vec{\alpha}}(q_k)| = |\tau_{\vec{\alpha}}(p_k^{\perp}) - \tau_{\vec{\alpha}}(q_k^{\perp})| < \epsilon_k(\vec{\alpha}),
		\]
		which forces $\tau_{\vec{\alpha}}(p_k) = \tau_{\vec{\alpha}}(q_k)$ by definition of $\epsilon_k(\vec{\alpha})$.  Now let $p_{m,n}$ and $q_{m,n}$ be the components of $p$ and $q$ respectively in the direct summand $(M_m(\C),\alpha_{m,n})$. 
		Because the $\alpha_{m,n}$'s are linearly independent over $\Q$, the condition that $\tau_{\vec{\alpha}}(p_k) = \tau_{\vec{\alpha}}(q_k)$ forces that $\tr_m(p_{m,n}) = \tr_m(q_{m,n})$ for $m, n \leq k$.  Because $\vec{\alpha} \in G$, we know that $\vec{\alpha} \in G_k$ for infinitely many $k$, and thus $\tr_m(p_{m,n}) = \tr_m(q_{m,n})$ for all $m, n$, which means that $p$ and $q$ are conjugate by an automorphism.  Therefore, by Theorem \ref{thm: main QE}, $\Th(\cM_{\vec{\alpha}})$ admits quantifier elimination.
	\end{proof}

	\subsection{Matrix amplification and approximate embedding} \label{sec: amplification}
	
	In Theorem \ref{thm: main MC}, we assumed the condition that $M_2(\cM)$ embeds into $\cM^{\cU}$.  While this condition holds automatically if $\cM$ is Connes-embeddable or if $\cM$ is existentially closed, we do not know if it holds for all $\mathrm{II}_1$ factors.  In this section, we investigate this problem by giving a series of equivalent conditions.  This expands upon the results about the ``universal fundamental group'' by Goldbring and Hart \cite[Proposition 4.17]{GH2017}.\footnote{The reader should be warned that in this proposition, clauses (1) and (2) should start with `For any II$_1$ factor $\cM$,\dots’.} 
	
	Recall that for $\mathrm{II}_1$ factors $\cM$ and $\cN$, the statement $\Th_\exists(\cM) = \Th_\exists(\cN)$ means that for every $\inf$-sentence $\phi$, we have $\phi^{\cM} = \phi^{\cN}$.  An equivalent statement is that for some ultrafilter $\cU$, we have that $\cM$ embeds into $\cN^{\cU}$ and $\cN$ embeds into $\cM^{\cU}$.  For instance, when $\cM$ is Connes-embeddable, then $\Th_\exists(\cM) = \Th_\exists(\cR)$.  We will show that the condition of $M_2(\cM)$ embedding into $\cM^{\cU}$ is equivalent to $\Th_\exists(\cM^t) = \Th_\exists(\cM)$ for some or all $t \in (0,\infty) \setminus \{1\}$, where $\cM^t$ is the $t$th compression/amplification of $\cM$.
	
	\begin{prop} \label{prop: amplification limit}
		Let $\cM$ be a $\mathrm{II}_1$ factor.Then
		\begin{align*}
			\lim_{t \to \infty} \Th_\exists(\cM^t) &= \Th_\exists(\cM \otimes \cR), \\
			\lim_{t \to 0} \Th_\exists(\cM^t) & \text{ exists.}
		\end{align*}
	\end{prop}
	
	\begin{proof}
		Consider an existential sentence $\phi = \inf_{x_1,\dots,x_n} \psi(x_1,\dots,x_n)$ where $\psi$ is a quantifier-free formula and $x_j$ ranges over the unit ball.  We can express
		\begin{equation} \label{eq: inf-formula expression}
			\psi(\mathbf{x}) = F(\re \tr(p_1(\mathbf{x})), \dots, \re \tr(p_k(\mathbf{x})))
		\end{equation}
		for some non-commutative $*$-polynomials $p_j$ and a continuous real-valued function $F$.  By rescaling the input variables to $F$, assume without loss of generality that $\norm{p_j(\mathbf{x})} \leq 1$ when $x_1$, \dots, $x_n$ are in the unit ball.  Let $\omega_F$ be the modulus of continuity of $F$ with respect to the $\ell^\infty$-norm on $[-1,1]^k$.  Suppose that $s < t$.  Write
		\[
		t = ms + \epsilon, \text{ where } m \in \N \text{ and } \epsilon \in [0,t/s).
		\]
		Let $\iota_{s,t}: \cM^s \to \cM^t$ be the non-unital $*$-homomorphism $\iota_{s,t}(x) = x^{\oplus m} \oplus 0_{\cM^\epsilon}$, and note that
		\[
		|\tr^{\cM^t}(\iota_{s,t}(y)) - \tr^{\cM^s}(y)| \leq \frac{\epsilon}{t} \norm{y} \leq \frac{s}{t} \norm{y}.
		\]
		Hence, from \eqref{eq: inf-formula expression} and the uniform continuity of $F$,
		\[
		\psi^{\cM^t}(\iota_{s,t}(\mathbf{x})) \leq \psi^{\cM^s}(\mathbf{x}) + \omega_F(s/t), \text{ hence } \phi^{\cM^t} \leq \phi^{\cM^s} + \omega_F(s/t).
		\]
		For each $s \in (0,\infty)$, we have
		\[
		\limsup_{t \to \infty} \phi^{\cM^t} \leq \liminf_{t \to \infty} \left[ \phi^{\cM^s} + \omega_F(s/t) \right] = \phi^{\cM^s}.
		\]
		Since $s$ was arbitrary, it follows that $\lim_{t \to \infty} \phi^{\cM^t} = \inf_{t \in (0,\infty)} \phi^{\cM^t}$.  A similar argument shows that $\lim_{t \to 0^+} \phi^{\cM^t} = \sup_{t \in (0,\infty)} \phi^{\cM^t}$.
		It remains to show that the limit as $t \to \infty$ agrees with $\Th_\exists(\cM \otimes \cR)$.  First, note that $\cM$ embeds into $\cM \otimes \cR$, so also $\cM^t$ embeds into $(\cM \otimes \cR)^t = \cM \otimes \cR^t \cong \cM \otimes \cR$.  Thus, for each $\inf$-sentence $\phi$, we have $\phi^{\cM \otimes \cR} \leq \lim_{t \to \infty} \phi^{\cM^t}$.  For the opposite inequality, note that $\cM \otimes \cR$ embeds into $\cN = \prod_{n \to \cU} \cM \otimes M_n(\C)$.
	\end{proof}
	
	\begin{prop}
		Let $\cM$ be a $\mathrm{II}_1$ factor.  Then the following are equivalent:
		\begin{enumerate}[(1)]
			\item $M_2(\cM)$ embeds into $\cM^{\cU}$ for some ultrafilter $\cU$.
			\item $\alpha \cM \oplus (1 - \alpha) \cM$ embeds into $\cM^{\cU}$ for some ultrafilter $\cU$ and some $\alpha \in (0,1)$.
			\item $\Th_\exists(\cM) = \Th_\exists(\cM \otimes \cR)$.
			\item $\Th_\exists(\cM^t) = \Th_\exists(\cM)$ for all $t > 0$.
			\item $\Th_\exists(\cM^t) = \Th_\exists(\cM)$ for some $t \neq 1$.
			\item $\lim_{t \to \infty} \Th_\exists(\cM^t) = \lim_{t \to 0} \Th_\exists(\cM^t)$.
			\item There exists a McDuff $\mathrm{II}_1$ factor $\cN$ such that $\Th_\exists(\cM) = \Th_\exists(\cN)$.
			\item There exists a Gamma $\mathrm{II}_1$ factor $\cN$ such that $\Th_\exists(\cM) = \Th_\exists(\cN)$.
		\end{enumerate}
	\end{prop}
	
	\begin{proof}
		(1) $\implies$ (2) because $(1/2) \cM \oplus (1/2) \cM$ is contained in $M_2(\cM)$.
		
		(2) $\implies$ (3).  Let $\iota: \alpha \cM \oplus (1 - \alpha) \cM \to \cM^{\cU}$ be an embedding where $\cU$ is an ultrafilter on index set $I$.  Let $p = \iota(1 \oplus 0)$.  Let $\Delta: \cM \to \alpha \cM \oplus (1 - \alpha) \cM$ be the diagonal map.  Then $\Delta(\cM)$ commutes with $p$ and hence $\operatorname{Ad}_p \circ \iota \circ \Delta$ gives an embedding $\cM \to p(\cM^{\cU})p$ since $\cM$ is a $\mathrm{II}_1$ factor.  Now $p$ lifts to a family of projections $(p_i)_{i \in I}$ with $\tr^{\cM}(p_i) = \tr^{\cM^{\cU}}(p) = \alpha$.  Since $p_i$ is unitarily conjugate to some fixed projection $p_0 \in \cM$ for all $i$, we have $p \cM^{\cU} p = \prod_{i \to \cU} p_i \cM p_i = (p_0 \cM p_0)^{\cU}$.  In other words, $\cM$ embeds into an ultraproduct of $\cM^{\alpha}$.  This also implies that $\cM^t$ embeds into an ultraproduct of $\cM^{t\alpha}$ for each $t \in (0,\infty)$.  Hence, $\cM^{1/\alpha^k}$ embeds into an ultraproduct of $\cM$ for each $k \in \cN$.  Thus, for an $\inf$-formula $\phi$,
		\[
		\phi^{\cM \otimes \cR} = \lim_{t \to \infty} \phi^{\cM^t} = \lim_{k \to \infty} \phi^{\cM^{1/\alpha^k}} \leq \phi^{\cM} \leq \phi^{\cM \otimes \cR}.
		\]
		Hence, $\Th_{\exists}(\cM) = \Th_{\exists}(\cM \otimes \cR)$.
		
		(3) $\implies$ (1).  Note $M_2(\cM)$ embeds into $\cM \otimes \cR$, which embeds into $\cM^{\cU}$.
		
		(3) $\iff$ (4).  When (3) holds, $\cM$ and $\cM \otimes \cR$ are embeddable into each other's ultrapowers, which implies that $\cM^t$ and $(\cM \otimes \cR)^t \cong \cM \otimes \cR^t \cong \cM \otimes \cR$ are embeddable into each other's ultrapowers.  Hence, $\Th_{\exists}(\cM^t) = \Th_{\exists}(\cM \otimes \cR) = \Th_{\exists}(\cM)$ for all $t \in (0,\infty)$.  Conversely, if $\Th_{\exists}(\cM) = \Th_{\exists}(\cM^t)$ for all $t$, then we have $\Th_{\exists}(\cM \otimes \cR) = \lim_{t \to \infty} \Th_{\exists}(\cM^t) = \Th_{\exists}(\cM)$.
		
		(4) $\implies$ (5) is immediate.
		
		(5) $\implies$ (6).  As in the proof of Proposition \ref{prop: amplification limit} or in (2) $\implies$ (3), since $\cM^t$ and $\cM$ embed into each other's ultrapowers, the same holds for $\cM^{t^k}$ for each $k \in \Z$, which implies (6).
		
		(6) $\implies$ (4).  This follows immediately from the fact that for any $\inf$-sentence $\phi$, we have $\lim_{t \to \infty} \phi^{\cM^t} = \inf_{t \in (0,\infty)} \phi^{\cM^t}$ and $\lim_{t \to 0} \phi^{\cM^t} = \sup_{t \in (0,\infty)} \phi^{\cM^t}$, which we showed in the proof of Proposition \ref{prop: amplification limit}.
		
		(3) $\implies$ (7) $\implies$ (8) is immediate by definition.
		
		(8) $\implies$ (2).  By assumption $\cM$ embeds into $\cN^{\cU}$.  Since $\cN$ has property Gamma, there exists a projection $p \in \cN^{\cU}$ that commutes with the image of $\cM$ (provided that ultrafilter $\cU$ is on a sufficiently large index set).  Then $\cM$ and $p$ generate a copy of $\alpha \cM \oplus (1 - \alpha) \cM$ in $\cN^{\cU}$, where $\alpha = \tr^{\cN^{\cU}}(p)$.  Finally, $\cN^{\cU}$ embeds into $\cM^{\cV}$ for some ultrafilter $\cV$, hence (2) holds.
	\end{proof}
	
	\begin{remark}
		By the usual arguments concerning countable saturation, if $\cM$ is separable, then it suffices to consider some or all free ultrafilters on $\N$ for conditions (1) and (2).
	\end{remark}
	
	\begin{remark}
		Similar reasoning shows that if $\cM^t$ embeds into $\cM^{\cU}$ for some $t > 1$, then $\cM \models \Th_{\exists}(\cM \otimes \cR)$, and hence $\Th_{\exists}(\cM) = \Th_{\exists}(\cM \otimes \cR)$.  Therefore, if these conditions fail, then $\cM^s$ does not embed into $(\cM^t)^{\cU}$ for any $s > t$.  Thus, all the existential theories of $\cM^t$ for $t \in \R_+$ are distinct and the first-order fundamental group is trivial.
		
		Compare \cite[Proposition 4.16]{GH2017} which showed that if the first-order fundamental group of $\cM$ is not all of $\R_+$, then it is countable and hence there are continuum many non elementary equivalent matrix amplifications of $\cM$.  The same argument of course applies to the fundamental group for the existential theory.  Note also from \cite[Proposition 5.1]{GH2016} that a negative solution to Connes embedding immediately implies the existence of continuum many existential theories of type $\mathrm{II}_1$ algebras (but not factors).
	\end{remark}
	
	\subsection{The non-tracial setting} \label{sec: nontracial}
	
	What major elementary classes of self-adjoint operator algebras admit quantifier elimination? The question for \cstar-algebras (both unital and non-unital) has been resolved in \cite{EFKV2017} and the results of the present paper, together with \cite{Farah2023}, resolve the question in case of tracial von Neumann algebras. What remains is the case of von Neumann algebras with arbitrary faithful normal states, in particular type $\mathrm{III}$ von Neumann algebras.  Metric languages for the non-tracial setting were given in \cite{Dabrowski2019,AGHS2025}; see \cite{AH2014,Ando2023} for ultraproducts in the non-tracial setting.
	
	For non-tracial factors, quantifier elimination and model completeness can depend on the choice of state.  For instance, on 
	$M_3(\C)$ consider the state $\phi(A) = \tr_3(AH)$ where $H = \diag(h_1,h_2,h_3)$ with $h_1 > h_2 > h_3$.  Let $t \in (0,1)$ such that $h_2 = th_1 + (1-t)h_3$, and let
	\[
	P = \begin{bmatrix}
		t & 0 & t^{1/2}(1-t)^{1/2} \\
		0 & 0 & 0 \\
		t^{1/2}(1-t)^{1/2} & 0 & 1 - t
	\end{bmatrix}, \qquad Q = \begin{bmatrix}
		0 & 0 & 0 \\
		0 & 1 & 0 \\
		0 & 0 & 0
	\end{bmatrix}.
	\]
	Then $P$ and $Q$ are projections and $\phi(P) = \phi(Q)$ but they are not conjugate by a state-preserving automorphism of $M_3(\C)$.  Hence, the theory of  $(M_3(\C),\phi)$ does not admit quantifier elimination.
	
	However, in the type $\mathrm{III}_1$ setting, the Connes-St{\o}rmer transitivity theorem \cite{CS1978} implies that all states are approximately unitarily equivalent, and hence for any two states the associated Ocneanu ultraproducts $(\cM,\phi)^{\cU}$ and $(\cM,\psi)^{\cU}$ are isomorphic, and so $(\cM,\phi)$ and $(\cM,\psi)$ are elementarily equivalent.  In fact, we believe the random matrix argument given here likely will adapt to the type $\mathrm{III}_1$ setting.  Indeed, let $\mathrm{T}$ be the theory of some type $\mathrm{III}_1$ factor $(\cM,\phi)$.  Since $\cM$ is type $\mathrm{III}$, we have $\cM \cong \cM \otimes M_n(\C)$.  Thus, the ultraproduct $(\cN,\psi) = \prod_{n \to \cU} (M_n(\C),\tr_n) \otimes (\cM,\phi)$ is a model of $\mathrm{T}$.  The random matrix construction of \S \ref{sec: MC factor proof} yields two elements $\mathbf{X}$ and $\mathbf{Y}$ in this ultraproduct such that $f^{\cN,\psi}(\mathbf{Y}) \leq f^{\cN,\psi}(\mathbf{X})$ for inf-formulas $f$, $\{\mathbf{X}\}'$ and $\{\mathbf{Y}\}'$ are definable sets with respect to parameters $\mathbf{X}$ and $\mathbf{Y}$ respectively,\footnote{Technically, one has to check that appropriate sets of left/right bounded elements in the commutant are definable sets, which could require a small additional argument.} and $\{\mathbf{X}\}'$ is a $\mathrm{III}_1$ factor and $\{\mathbf{Y}\}'$ is not.  Because $\mathrm{III}_1$ factors are an axiomatizable class \cite[Proposition 6.5.7]{GHScorrespondences}, this means that $\mathbf{X}$ and $\mathbf{Y}$ cannot have the same type.
	
	In the type $\mathrm{III}_\lambda$ setting for $\lambda \in (0,1)$, we do not know if this argument goes through because we would have to pay more attention to the choice of state, and the random matrix argument requires having models with a tensor product decomposition $(M_n(\C),\tr_n) \otimes (\cM,\phi)$. In the type $\mathrm{III}_0$ and type $\mathrm{II}_\infty$ setting, another issue arises, namely that type $\mathrm{III}_0$ and type $\mathrm{II}_\infty$ factors are not axiomatizable classes \cite[Corollary 8.6 and Proposition 8.3]{AGHS2025}, so examining factoriality of the relative commutant of $\mathbf{X}$ and $\mathbf{Y}$ may not distinguish their types.  Likely, a different approach is needed in these cases.

	\appendix
	
	\section{Model completeness and inf-formulas}
	
	This section proves the characterization of model completeness for theories of metric structure in terms of types and formulas.
	
	\begin{lem} \label{lem: type MC 2}
		Let $\mathrm{T}$ be an $\Lang$-theory.  Then the following are equivalent:
		\begin{enumerate}
			\item $\mathrm{T}$ is model complete, i.e.\ if $\cM$ and $\cN$ are models of $\mathrm{T}$, then every embedding $\cM \to \cN$ of $\Lang$-structures is an elementary embedding.
			\item For every $n$ and every pair $\mu, \nu \in \mathbb{S}_n(\mathrm{T})$, if $\psi(\mu) \leq \psi(\nu)$ for every $\inf$-formula $\psi$, then $\mu = \nu$.
			\item For every $\Lang$-formula $\phi$ and $\epsilon > 0$, there exists an inf-formula $\psi$ such that $|\phi - \psi| < \epsilon$ (on the appropriate sort or domain) for all models of $\mathrm{T}$.
		\end{enumerate}
	\end{lem}

	\begin{proof}
		(3) $\implies$ (1).  Assume that (3) holds.  Let $\cM \to \cN$ be an inclusion of models of $\mathrm{T}$.  Let $\phi$ be an $n$-variable formula and let $\mathbf{x} = (x_1,\dots,x_n)$ be a tuple of the appropriate sort from $\cM$.  Let $\epsilon > 0$.  Then by (3), there exist inf-formulas $\psi_1$ and $\psi_2$ such that $|\psi_1 - \phi| < \epsilon$ and $|\psi_2 - (-\phi)| < \epsilon$ in all models of $\mathrm{T}$.  In particular,
		\[
		\phi^{\cN}(\mathbf{x}) \leq \psi_1^{\cN}(\mathbf{x}) + \epsilon \leq \psi_1^{\cM}(\mathbf{x}) + \epsilon \leq \phi^{\cM}(\mathbf{x}) + 2 \epsilon,
		\]
		and symmetrically $-\phi^{\cN}(\mathbf{x}) \leq -\phi^{\cM}(\mathbf{x}) + 2 \epsilon$.  Since $\epsilon$ was arbitrary, we have  $\phi^{\cM}(\mathbf{x}) = \phi^{\cN}(\mathbf{x})$, so the embedding $\cM \to \cN$ is elementary.
		
		(1) $\implies$ (2).  Suppose $\mathrm{T}$ is model complete.  Let $\mu$ and $\nu$ be $n$-types satisfying the hypothesis for (2).  Let $\kappa$ be the density character of $\Lang$, and fix a $\kappa^+$-saturated model $\cM$ of $\mathrm{T}$.  Then $\cM$ contains some $\mathbf{x}$ with type $\mu$ and some $\mathbf{y}$ with type $\nu$.  By the downward L{\"o}wenheim-Skolem theorem \cite[Proposition 7.3]{BYBHU2008}, there exists an elementary substructure $\cN \preceq \cM$ containing $\mathbf{y}$ with density character at most $\kappa$.  Let $\mathbf{z}$ be a family indexed by some  set $I$ of cardinality $\kappa$ that is dense in $\cN$.  For every finite $F \subseteq I$, every $k\geq 1$,  and every $k$-tuple of  quantifier-free formulas $\phi_1, \dots, \phi_k$ in $n + |F|$ variables, consider the formula
		\[
		\psi(u_1,\dots,u_n) = \inf_{(v_i)_{i \in F}} \max_{j=1,\dots,k} |\phi_j(u_1,\dots,u_n, (v_i)_{i \in F}) - \phi_j^{\cM}(y_1,\dots,y_n,(z_i)_{i \in F})|.
		\]
		By assumption $\psi^{\cM}(x_1,\dots,x_n)\leq \psi^{\cM}(y_1,\dots,y_n)  = 0$.  Therefore, for any $\epsilon > 0$, there exists $(w_i)_{i \in F}$ such that $|\phi_j^{\cM}(y_1,\dots,y_n,(w_i)_{i \in F}) - \phi_j^{\cM}(x_1,\dots,x_n,(z_i)_{i \in F})| < \epsilon$ for all $j = 1, \ldots, k$.  By saturation, this implies that there exists a family  $\mathbf{w}$ indexed by $I$ in $\cM$ such that $(\mathbf{x},\mathbf{w})$ has the same quantifier-free type as $(\mathbf{y},\mathbf{z})$.  In particular, the substructure $\tilde{\cN}$ of $\cM$ generated by $(\mathbf{x},\mathbf{w})$ is isomorphic to the substructure $\cN$ generated by $(\mathbf{y},\mathbf{z})$.  So $\tilde{\cN}$ is a model of $\mathrm{T}$ and by model completeness the inclusion $\tilde{\cN} \to \cM$ is elementary.  Therefore,
		\[
		\tp^{\cM}(\mathbf{x}) = \tp^{\tilde{\cN}}(\mathbf{x}) = \tp^{\cN}(\mathbf{y}) = \tp^{\cM}(\mathbf{y}),
		\]
		and $\mu = \nu$ as desired.

		(2) $\implies$ (3).  Our argument uses point-set topology on $\mathbb{S}_n(\mathrm{T})$ and is motivated by Urysohn's lemma and the Stone--Weierstrass theorem.
		
		\textbf{Claim 1:} \emph{For every type $\mu$ and neighborhood $\mathcal{O}$ of $\mu$, there exist inf-formulas $\psi_1,\dots, \psi_k$ and $\delta > 0$ such that for types $\nu$, if $\psi_j(\nu) > \psi_j(\mu) - \delta$ for $j = 1, \dots k$, then $\nu \in \mathcal{O}$.}
		
		Fix $\mu$ and a neighborhood $\mathcal{O}$, and suppose for contradiction that no such inf-formulas exist.  Then for every $\delta > 0$ and every finite collection of inf-formulas $\psi_1$, \dots, $\psi_k$, there exists some type $\nu \in \mathbb{S}_n(\mathrm{T}) \setminus \mathcal{O}$ satisfying $\psi_j(\nu) > \psi_j(\mu) - \delta$ for $j = 1, \dots k$.  Since $\mathbb{S}_n(\mathrm{T}) \setminus \mathcal{O}$ is compact, there exists some $\nu \in \mathbb{S}_n(\mathrm{T}) \setminus \mathcal{O}$ satisfying $\psi(\nu) \geq \psi(\mu)$ for all inf-formulas $\phi$.  By (3), this implies $\nu = \mu$, which contradicts $\nu \in \mathbb{S}_n(\mathrm{T}) \setminus \mathcal{O}$.
		
		\textbf{Claim 2:} \emph{For every type $\mu$ and neighborhood $\mathcal{O}$, there exists an inf-formula $\psi$ taking values in $[0,1]$ such that $\psi(\mu)>0$ and, for all types $\nu$, if $\psi(\nu) > 0$, then $\nu \in \mathcal{O}$.}
		
		Let $\psi_1$, \dots, $\psi_k$ and $\delta$ be as in Claim 1, and set
		\[
		\psi = \min_j \max(\psi_j - \psi_j(\mu) + \delta, 0),
		\]
		which is an inf-formula by the monotonicity of $\max$ and $\min$.
		
		\textbf{Claim 3:} \emph{Let $\mathcal{E}_0$ and $\mathcal{E}_1$ be disjoint closed subsets of $\mathbb{S}_n(\mathrm{T})$.  Then there exists an inf-formula $\psi$ taking values in $[0,1]$ such that $\psi|_{\mathcal{E}_0} = 0$ and $\psi|_{\mathcal{E}_1} = 1$.}
		
		By Claim 2, for each $\mu \in \mathcal{E}_1$, there exists a nonnegative inf-formula $\psi_\mu$ such that $\psi_\mu(\mu)>0$ and if $\psi_\mu(\nu) > 0$, then $\nu \in \mathbb{S}_n(\mathrm{T}) \setminus \mathcal{E}_0$.  Let $\mathcal{O}_\mu = \{\nu: \psi_\mu(\nu) > 0\}$.  These neighborhoods form an open cover of the compact set $\mathcal{E}_1$, and hence $\mathcal{E}_1$ can be covered by finitely many of these neighborhoods, say $\mathcal{O}_{\mu_1}$, \dots, $\mathcal{O}_{\mu_k}$.  Thus, $\sum_{j=1}^k \psi_j$ is strictly positive on $\mathcal{E}_1$ and attains some minimum $\delta > 0$ on this set.  Then
		\[
		\psi = \min \left( 1, \frac{1}{\delta} \sum_{j=1}^k \psi_j \right)
		\]
		is an inf-formula with the desired properties.
		
		\textbf{Claim 4:} \emph{For every formula $\phi$ and $\epsilon > 0$, there exists an inf-formula $\psi$ such that $|\phi - \psi| < \epsilon$ in every model of $\mathrm{T}$.}
		
		By affine transformation, assume without loss of generality that $0 \leq \phi \leq 1$.  Let $k \in \N$ with $1/k < \epsilon$.  For $j = 1$, \dots, $k$,  the sets $\{\phi \leq (j-1)/k\}$ and $\{ \phi \geq j/k\}$ are disjoint and closed in $\mathbb{S}_n(\mathrm{T})$.  By Claim 3, there exists an inf-formula $\psi_j$ such that $0 \leq \psi_j \leq 1$ and for $\nu \in \mathbb{S}_n(\mathrm{T})$,
		\[
		\phi(\nu) \leq (j-1)/k \implies \psi_j(\nu) = 0, \qquad \phi(\nu) \geq j/k \implies \psi_j(\nu) = 1.
		\]
		Let
		\[
		\psi = \frac{1}{k} \sum_{j=1}^k \psi_j.
		\]
		Then for types $\nu$, if $\phi(\nu) \in [(j-1)/k,j/k]$, then $\psi_1(\nu)$, \dots, $\psi_{j-1}(\nu)$ are $1$ and $\psi_{j+1}(\nu)$, \dots, $\psi_k(\nu)$ are zero, so that $\psi(\nu) \in [(j-1)/k,j/k]$.  Hence, $|\phi(\nu) - \psi(\nu)| \leq 1/k < \epsilon$ for all $\nu \in \mathbb{S}_n(\mathrm{T})$ as desired.
	\end{proof}
	
	\subsection*{Funding}
	
	IF and DJ were partially supported by the Discovery grant ``Logic and $\mathrm{C}^*$-algebras'' from the Natural Sciences and Engineering Research Council (Canada). JP was partially supported by the National Science Foundation (US), grant DMS-2054477. DJ acknowledges support from the Danish Independent Research Fund, grant 1026-00371B, and the Horizon Europe Marie Sk{\l}odowksa-Curie Action FREEINFOGEOM, grant 101209517.

	\bibliographystyle{abbrv}
	\bibliography{QEMC}

\begin{thebibliography}{10}

\bibitem{AS2004}
A.~Ambainis and A.~Smith.
\newblock Small pseudo-random families of matrices: Derandomizing approximate
  quantum encryption.
\newblock In K.~Jansen, S.~Khanna, J.~D.~P. Rolim, and D.~Ron, editors, {\em
  Approximation, Randomization, and Combinatorial Optimization. Algorithms and
  Techniques}, pages 249--260, Berlin, Heidelberg, 2004. Springer Berlin
  Heidelberg.

\bibitem{AGZ2009}
G.~W. Anderson, A.~Guionnet, and O.~Zeitouni.
\newblock {\em An Introduction to Random Matrices}.
\newblock Cambridge Studies in Advanced Mathematics. Cambridge University
  Press, 2009.

\bibitem{Ando2023}
H.~Ando.
\newblock Introduction to nontracial ultraproducts of von {N}eumann algebras.
\newblock In I.~Goldbring, editor, {\em Model Theory of Operator Algebras},
  pages 303--341. DeGruyter, Berlin, Boston, 2023.

\bibitem{AH2014}
H.~Ando and U.~Haagerup.
\newblock Ultraproducts of von {N}eumann algebras.
\newblock {\em J. Funct. Anal.}, 266:6842--6913, 2014.

\bibitem{AGHS2025}
J.~Aruseelan, I.~Goldbring, B.~Hart, and T.~Sinclair.
\newblock Totally bounded elements in $\mathrm{W}^*$-probability spaces.
\newblock Preprint, arXiv:2501.14153, 2025.

\bibitem{AGKE2022}
S.~Atkinson, I.~Goldbring, and S.~{Kunnawalkam Elayavalli}.
\newblock Factorial relative commutants and the generalized jung property for
  ii1 factors.
\newblock {\em Advances in Mathematics}, 396:108107, 2022.

\bibitem{BdlHVpropertyT}
B.~Bekka, P.~de~la Harpe, and A.~Valette.
\newblock {\em Kazhdan's property {(T)}}.
\newblock Cambridge University Press, 2008.

\bibitem{BAG1997}
G.~{Ben Arous} and A.~Guionnet.
\newblock Large deviations for {W}igner's law and {V}oiculescu's
  non-commutative entropy.
\newblock {\em Probab Theory Relat Fields}, 108:517--542, 1997.

\bibitem{BST2010}
A.~Ben-Aroya, O.~Schwartz, and A.~Ta-Shma.
\newblock Quantum expanders: motivation and constructions.
\newblock {\em Theory Comput}, 6:47--79, 2010.

\bibitem{BT2007}
A.~Ben-Aroya and A.~Ta-Shma.
\newblock Quantum expanders and the quantum entropy difference problem.
\newblock preprint, arXiv:quant-ph/0702129, 2007.

\bibitem{BY2012}
I.~{Ben Yaacov}.
\newblock On theories of random variables.
\newblock {\em Israel J. Math.}, 194(2):957--1012, 2013.

\bibitem{BYBHU2008}
I.~{Ben Yaacov}, A.~Berenstein, C.~W. Henson, and A.~Usvyatsov.
\newblock Model theory for metric structures.
\newblock In Z.~C. et~al., editor, {\em Model Theory with Applications to
  Algebra and Analysis, Vol. II}, volume 350 of {\em London Mathematical
  Society Lecture Notes Series}, pages 315--427. Cambridge University Press,
  2008.

\bibitem{BYIT2024}
I.~{Ben Yaacov}, T.~Ibarluc{\'\i}a, and T.~Tsankov.
\newblock Extremal models and direct integrals in affine logic.
\newblock Preprint arXiv:2407.13344, 2024.

\bibitem{BYU2010}
I.~{Ben Yaacov} and A.~Usvyatsov.
\newblock Continuous first order logic and local stability.
\newblock {\em Transactions of the American Mathematical Society},
  362(10):5213--5259, 10 2010.

\bibitem{BH2023}
A.~Berenstein and C.~W. Henson.
\newblock Model-theory of probability spaces.
\newblock In I.~Goldbring, editor, {\em Model Theory of Operator Algebras},
  pages 159--213. DeGruyter, Berlin, Boston, 2023.

\bibitem{Blackadar2006}
B.~Blackadar.
\newblock {\em Operator Algebras: Theory of ${C}^*$-algebras and von {N}eumann
  algebras}, volume 122 of {\em Encyclopaedia of Mathematical Sciences}.
\newblock Springer-Verlag, Berlin, Heidelberg, 2006.

\bibitem{BC2022}
C.~Bordenave and B.~Collins.
\newblock Strong asymptotic freeness for independent uniform variables on
  compact groups associated to nontrivial representations.
\newblock preprint, arXiv:2012.08759, 2020.

\bibitem{Bro:Topological}
N.~P. Brown.
\newblock Topological dynamical systems associated to {$II_1$} factors.
\newblock {\em Adv. Math.}, 227(4):1665--1699, 2011.
\newblock With an appendix by Narutaka Ozawa.

\bibitem{chang1990model}
C.~C. Chang and H.~J. Keisler.
\newblock {\em Model Theory}, volume~73 of {\em Studies in Logic and the
  Foundations of Mathematics}.
\newblock North-Holland Publishing Co., third edition, 1990.

\bibitem{CS1978}
A.~Connes and E.~St{\o}rmer.
\newblock Homogeneity of the state space of factors of type {III$_1$}.
\newblock {\em Journal of Functional Analysis}, 28(2):187--196, 1978.

\bibitem{Dabrowski2019}
Y.~Dabrowski.
\newblock Continuous model theories for von {N}eumann algebras.
\newblock {\em Journal of Functional Analysis}, 277(11):108308, 2019.

\bibitem{Davidson1996}
K.~R. Davidson.
\newblock {\em $C^*$-algebras by example}, volume~6 of {\em Fields Institute
  Monographs}.
\newblock American Mathematical Society, Providence, 1996.

\bibitem{Dixmier1969}
J.~Dixmier.
\newblock {\em Les ${C}^*$-alg{\`e}bres et leurs repr{\'e}sentations},
  volume~29 of {\em Cahiers Scientifiques}.
\newblock Gauthier-Villars, Paris, 2 edition, 1969.
\newblock Reprinted by {E}ditions {J}acques {G}abay, {P}aris, 1996.
  {T}ranslated as ${C}^*$-algebras, North-Holland, Amsterdam, 1977. First Edi-
  tion 1964.

\bibitem{EFKV2017}
C.~J. Eagle, I.~Farah, E.~Kirchberg, and A.~Vignati.
\newblock Quantifier elimination in $\mathrm{C}^*$-algebras.
\newblock {\em International Mathematics Research Notices},
  2017(24):7580--7606, 11 2016.

\bibitem{Fa:Logic}
I.~Farah.
\newblock Logic and operator algebras.
\newblock In S.~Y. Jang et~al., editors, {\em Proceedings of the Seoul ICM,
  volume II}, pages 15--39. Kyung Moon SA, 2014.

\bibitem{Fa:STCstar}
I.~Farah.
\newblock {\em Combinatorial Set Theory and \cstar-algebras}.
\newblock Springer Monogr. Math. Springer, 2019.

\bibitem{Farah2023}
I.~Farah.
\newblock Quantifier elimination in {II$_1$} factors.
\newblock preprint, arXiv:2304.11371, 2023.

\bibitem{FG2023}
I.~Farah and S.~Ghasemi.
\newblock Preservation of elementarity by tensor products of tracial von
  neumann algebras.
\newblock {\em Pacific J. Math.}, 332(1):91--113, 2024.

\bibitem{FHLRTVW2021}
I.~Farah, B.~Hart, M.~Lupini, L.~Robert, A.~Tikuisis, A.~Vignati, and
  W.~Winter.
\newblock Model theory of $\mathrm{C}^*$-algebras.
\newblock {\em Memoirs of the American Mathematical Society}, 271(1324), 2021.

\bibitem{FHS2013}
I.~Farah, B.~Hart, and D.~Sherman.
\newblock Model theory of operator algebras {I}: stability.
\newblock {\em Bulletin of the London Mathematical Society}, 45(4):825--838,
  2013.

\bibitem{FHS2014}
I.~Farah, B.~Hart, and D.~Sherman.
\newblock Model theory of operator algebras {II}: model theory.
\newblock {\em Israel Journal of Mathematics}, 201(1):477--505, 2014.

\bibitem{FHS2014b}
I.~Farah, B.~Hart, and D.~Sherman.
\newblock Model theory of operator algebras {III}: elementary equivalence and
  $\mathrm{II}_1$ factors.
\newblock {\em Bulletin of the London Mathematical Society}, 46(3):609--628,
  2014.

\bibitem{farah2018omitting}
I.~Farah and M.~Magidor.
\newblock Omitting types in logic of metric structures.
\newblock {\em J. Math. Logic Logic}, 18(02):1850006, 2018.

\bibitem{GaoJekelIntegral}
D.~Gao and D.~Jekel.
\newblock Elementary equivalence and disintegration of tracial von {N}eumann
  algebras.
\newblock Preprint, arXiv:2410.05529, to appear in Forum of Math Sigma, 2024.

\bibitem{Goldbring2023spectralgap}
I.~Goldbring.
\newblock Spectral gap and definability.
\newblock In J.~Iovino, editor, {\em Beyond First Order Model Theory, Volume
  II}, page~36. Chapman and Hall/CRC, 2023.

\bibitem{GH2016}
I.~Goldbring and B.~Hart.
\newblock Computability and the {Connes} embedding problem.
\newblock {\em The Bulletin of Symbolic Logic}, 22(2):238--248, 2016.

\bibitem{GH2017}
I.~Goldbring and B.~Hart.
\newblock On the theories of mcduff's {II$_1$} factors.
\newblock {\em International Mathematics Research Notices}, 2017:5609--5628,
  2017.

\bibitem{GH2023}
I.~Goldbring and B.~Hart.
\newblock A survey on the model theory of tracial von {N}eumann algebras.
\newblock In I.~Goldbring, editor, {\em Model Theory of Operator Algebras},
  pages 133--157. DeGruyter, Berlin, Boston, 2023.

\bibitem{GHScorrespondences}
I.~Goldbring, B.~Hart, and T.~Sinclair.
\newblock Correspondences, ultraproducts, and model theory.
\newblock preprint, arXiv:1809.00049.

\bibitem{GHS2013}
I.~Goldbring, B.~Hart, and T.~Sinclair.
\newblock The theory of tracial von {N}eumann algebras does not have a model
  companion.
\newblock {\em Journal of Symbolic Logic}, 78(3):1000--1004, 2013.

\bibitem{GJKEP2023}
I.~Goldbring, D.~Jekel, S.~{Kunnawalkam Elayavalli}, and J.~Pi.
\newblock Uniformly super {McDuff} {II$_1$} factors.
\newblock Preprint arXiv:2303.02809.

\bibitem{GE2008}
D.~Gross and J.~Eisert.
\newblock Quantum {M}argulis expanders.
\newblock {\em Quantum Information \& Computation}, 8(8):722--733, 2008.

\bibitem{Gross1975}
L.~Gross.
\newblock Logarithmic sobolev inequalities.
\newblock {\em American Journal of Mathematics}, 97(4):1061--1083, 1975.

\bibitem{Guionnet2009}
A.~Guionnet.
\newblock {\em Concentration inequalities for random matrices}, pages 65--87.
\newblock Springer, Berlin, Heidelberg, 2009.

\bibitem{Harrow2008}
A.~W. Harrow.
\newblock Quantum expanders from any classical cayley graph expander.
\newblock {\em Quantum Information \& Computation}, 8:715--721, 2008.

\bibitem{Hart2023}
B.~Hart.
\newblock An introduction to continuous model theory.
\newblock In I.~Goldbring, editor, {\em Model Theory of Operator Algebras},
  pages 83--131. DeGruyter, Berlin, Boston, 2023.

\bibitem{Hastings2007}
M.~B. Hastings.
\newblock Random unitaries give quantum expanders.
\newblock {\em Phys. Rev. A}, 76:032315, 9 2007.

\bibitem{HIKO2003}
C.~W. Henson and J.~Iovino.
\newblock Ultraproducts in analysis.
\newblock In C.~Finet and C.~Michaux, editors, {\em Analysis and Logic}, London
  Mathematical Society Lecture Note Series. Cambridge University Press, 2003.

\bibitem{Hirschfeld}
J.~Hirschfeld.
\newblock Finite forcing, existential types and complete types.
\newblock {\em The Journal of Symbolic Logic}, 45(1):93--102, 1980.

\bibitem{ioana2021almost}
A.~Ioana.
\newblock Almost commuting matrices and stability for product groups.
\newblock {\em arXiv preprint arXiv:2108.09589}, 2021.

\bibitem{Ioana2023}
A.~Ioana.
\newblock An introduction to von {N}eumann algebras.
\newblock In I.~Goldbring, editor, {\em Model Theory of Operator Algebras},
  pages 43--81. DeGruyter, Berlin, Boston, 2023.

\bibitem{JekelCoveringEntropy}
D.~Jekel.
\newblock Covering entropy for types in tracial $\mathrm{W}^*$-algebras.
\newblock {\em Journal of Logic and Analysis}, 15(2):1--68, 2023.

\bibitem{JekelModelEntropy}
D.~Jekel.
\newblock Free probability and model theory of tracial $\mathrm{W}^*$-algebras.
\newblock In I.~Goldbring, editor, {\em Model Theory of Operator Algebras},
  pages 215--267. DeGruyter, Berlin, Boston, 2023.

\bibitem{Jekel2024Optimal}
D.~Jekel.
\newblock Optimal transport for types and convex analysis for definable
  predicates in tracial $\mathrm{W}^*$-algebras.
\newblock {\em Journal of Functional Analysis}, 287(9):110583, 2024.

\bibitem{JNVWY2020}
Z.~Ji, A.~Natarajan, T.~Vidick, J.~Wright, and H.~Yuen.
\newblock {MIP*=RE}.
\newblock arXiv:2001.04383, 2020.

\bibitem{Jing2015}
N.~Jing.
\newblock Unitary and orthogonal equivalence of sets of matrices.
\newblock {\em Linear Algebra and its Applications}, 481:235--242, 2015.

\bibitem{JonesSunder1997}
V.~F. Jones and V.~S. Sunder.
\newblock {\em Introduction to Subfactors}.
\newblock London Mathematical Society Lecture Note Series. Cambridge University
  Press, 1997.

\bibitem{KadisonRingroseI}
R.~V. Kadison and J.~R. Ringrose.
\newblock {\em Fundamentals of the Theory of Operator Algebras {I}}, volume~15
  of {\em Graduate Studies in Mathematics}.
\newblock American Mathematical Society, Providence, 1983.

\bibitem{Meckes2019}
E.~S. Meckes.
\newblock {\em The Random Matrix Theory of the Classical Compact Groups}.
\newblock Cambridge Tracts in Mathematics. Cambridge University Press, 2019.

\bibitem{MvNROO4}
F.~Murray and J.~von Neumann.
\newblock On rings of operators {IV}.
\newblock {\em Ann. Math.}, 44:716--808, 1943.

\bibitem{PisierExpanders}
G.~Pisier.
\newblock Quantum expanders and geometry of operator spaces.
\newblock {\em J. Eur. Math. Soc.}, 16(6):1183--1219, 2014.

\bibitem{PopaVaes2022}
S.~Popa and S.~Vaes.
\newblock $\mathrm{W}^*$-rigidity paradigms for embeddings of $\mathrm{II}_1$
  factors.
\newblock {\em Commun. Math. Phys.}, 395:907--961, 2022.

\bibitem{Sakai1971}
S.~Sakai.
\newblock {\em $\mathrm{C}^*$-algebras and $\mathrm{W}^*$-algebras}, volume~60
  of {\em Ergebnisse der {M}athematik und ihrer {G}renzgebiete}.
\newblock Springer-Verlag, Berlin Heidelberg, 1971.

\bibitem{TakesakiI}
M.~Takesaki.
\newblock {\em Theory of Operator Algebras I}, volume 124 of {\em Encyclopaedia
  of Mathematical Sciences}.
\newblock Springer-Verlag, Berlin Heidelberg, 2002.

\bibitem{Voiculescu1991}
D.-V. Voiculescu.
\newblock Limit laws for random matrices and free products.
\newblock {\em Inventiones mathematicae}, 104(1):201--220, Dec 1991.

\bibitem{Voiculescu1998}
D.-V. Voiculescu.
\newblock A strengthened asymptotic freeness result for random matrices with
  applications to free entropy.
\newblock {\em International Mathematics Research Notices}, 1998(1):41--63,
  1998.

\bibitem{Zhu1993}
K.~Zhu.
\newblock {\em An Introduction to Operator Algebras}.
\newblock Studies in Advanced Mathematics. CRC Press, Ann Arbor, 1993.

\end{thebibliography}

\end{document}